\documentclass[reqno,11pt]{amsart}

\usepackage[T1,T2A]{fontenc}
\usepackage[utf8]{inputenc}

\usepackage{latexsym}
\usepackage{amssymb}
\usepackage{amsrefs}
\usepackage{bbm}
\usepackage[mathscr]{euscript}
\usepackage{nicefrac}

\usepackage{xcolor}
\usepackage{graphicx}
\usepackage{adjustbox}
\usepackage{caption}
\usepackage{tabularx}
\usepackage{tikz}
\usepackage{tikzscale}

\usetikzlibrary{
  calc,
  angles,
  quotes,
  positioning,
  automata,
  arrows,
  decorations.pathreplacing,
  fadings,
  3d
}

\usepackage[all]{xy}

\usepackage[normalem]{ulem}
\usepackage{blindtext}
\usepackage{comment}
\usepackage{stackengine}
\usepackage{enumitem}
\usepackage{subfiles}
\usepackage{xparse}

\stackMath

\usepackage{hyperref}

\NewDocumentCommand{\vdotss}{ O{0.35em} O{0.8pt} m }{%
  \begin{tikzpicture}[baseline={(0,-0.5ex)}]
    \pgfmathtruncatemacro{\lastdot}{#3-1}%
    \foreach \i in {0,...,\lastdot}{%
      \fill (0,{-\i*#1}) circle[radius=#2];
    }%
    \path (0,0.3ex) -- (0,{-\lastdot*#1-0.3ex});
  \end{tikzpicture}%
}

\newcommand{\eq}[2]{%
  \begin{equation}
    \label{eq:#1}
    #2
  \end{equation}%
}

\newcommand{\equ}[1]{\eqref{eq:#1}}

\newcommand{\hs}{homogeneous space}
\newcommand{\ggm}{G/\Gamma}
\newcommand{\hd}{Hausdorff dimension}

\newcommand{\da}{Diophantine approximation}
\newcommand{\di}{Diophantine}

\newcommand{\ignore}[1]{}

\DeclareMathOperator{\Lie}{Lie}
\DeclareMathOperator{\dist}{dist}
\DeclareMathOperator{\diam}{diam}

\DeclareMathOperator{\SL}{SL}
\DeclareMathOperator{\PSL}{PSL}

\DeclareMathOperator{\Ad}{Ad}
\DeclareMathOperator{\Ker}{Ker}
\DeclareMathOperator{\ad}{ad}
\DeclareMathOperator{\tr}{tr}
\DeclareMathOperator{\diag}{diag}
\DeclareMathOperator{\supp}{supp}

\newcommand{\vq}{\mathbf{q}}

\newcommand{\p}{\mathbf{p}}

\newcommand{\bv}{\boldsymbol{v}}
\newcommand{\bx}{\mathbf{x}}

\newcommand{\fa}{\mathfrak{a}}
\newcommand{\fe}{\mathfrak{e}}

\newcommand{\fg}{\mathfrak{g}}
\newcommand{\fh}{\mathfrak{h}}
\newcommand{\fu}{\mathfrak{u}}
\newcommand{\fv}{\mathfrak{v}}
\newcommand{\fs}{\mathfrak{s}}
\newcommand{\fc}{\mathfrak{c}}

\newcommand{\R}{\mathbb{R}}
\newcommand{\Z}{\mathbb{Z}}
\newcommand{\N}{\mathbb{N}}

\newcommand{\nz}{\smallsetminus\{0\}}
\newcommand{\mr}{M_{m,n}}
\newcommand{\amr}{\ensuremath{\Theta\in M_{m,n}}}

\newcommand{\ba}{\ensuremath{\mathbf{BA}_{m,n}}}
\newcommand{\baab}{\ensuremath{\mathbf{BA}_{\alf,\betf}}}

\newcommand{\wab}{\mathbf{W}_{\alf,\betf}}
\newcommand{\fab}{(f,\alf,\betf)}

\newcommand{\alf}{\vec{\alpha}}
\newcommand{\betf}{\vec{\beta}}

\theoremstyle{plain}
\newtheorem{theorem}{Theorem}[section]
\newtheorem{lemma}[theorem]{Lemma}
\newtheorem{corollary}[theorem]{Corollary}

\newtheorem{proposition}[theorem]{Proposition}

\theoremstyle{definition}
\newtheorem{definition}[theorem]{Definition}
\newtheorem{example}[theorem]{Example}

\theoremstyle{remark}
\newtheorem{remark}[theorem]{Remark}
\newtheorem{observation}[theorem]{Observation}

\numberwithin{equation}{section}

\renewcommand{\le}{\leqslant}
\renewcommand{\ge}{\geqslant}

\DeclareFixedFont{\got}{U}{euf}{m}{n}{14.4pt}

\makeatletter
\newcommand{\namedlabel}[2]{%
  \begingroup
    #2%
    \def\@currentlabel{#2}%
    \phantomsection
    \label{#1}%
  \endgroup
}
\makeatother

\title[Bounded trajectories of quasi-rays]{%
  Bounded trajectories of quasi-rays on \\ homogeneous spaces and
  Diophantine \\ approximation with weight functions%
}

\author[D. Kleinbock]{Dmitry Kleinbock}
\address{%
  Department of Mathematics,
  Brandeis University,
  Waltham, MA 02453, USA%
}
\email{kleinboc@brandeis.edu}

\author[V. Neckrasov]{Vasiliy Neckrasov}
\address{%
  Department of Mathematics,
  Brandeis University,
  Waltham, MA 02453, USA%
}
\email{vneckrasov@brandeis.edu}

\date{}

\thanks{The first-named author was supported by NSF grant DMS-2155111.}

\begin{document}

\begin{abstract} 
{Let $G$ be a connected semisimple real Lie group, $\Gamma$ an irreducible lattice in $G$ and $X = G/\Gamma$. Let $F = \{g_t: t\ge 0\}$ be a non-quasiunipotent one-parameter subsemigroup  of $G$.
{Then it is known that} the set of points in $X$ with bounded $F$-trajectories 
{has} full Hausdorff dimension. In addition, if  $U$ is the expanding horospherical subgroup relative to $g_1$, then for any $x \in X$ the set of points $u \in U$ such that the $F$-trajectory of $ux$ is bounded has full Hausdorff dimension. In this paper we take $U$ to be a horospherical subgroup of $G$ and  apply Shi's 
equidistribution theorem for elements of the expanding cone with respect to $U$   to describe a class of subsets $F$ in $G$, not presupposing the group structure, for which the above full \hd\  statements also hold. As an application, we prove that the set of badly approximable matrices in the set-up of Diophantine approximations with quasimultiplicative weight functions 
has full Hausdorff dimension.
}
\end{abstract}

\maketitle


\section{Introduction}

\subsection{Bounded trajectories on \hs s} \label{intro1} The study of bounded orbits on finite volume  \hs s of Lie groups dates back to the 1980s, when Dani used Schmidt's work on the winning property of the set  of badly approximable matrices to construct bounded trajectories in the space of unimodular lattices. Namely, 
let \eq{sld}{G = \SL_d(\R),\ \Gamma = \SL_d(\Z),\ X = X_d := \ggm} and 
\eq{defgt}{
F = \{g_t: t\ge 0\}, \text{ where } g_t: = \diag \left( e^{t/m}, \ldots, e^{t/m}, e^{-t/n}, \ldots, e^{-t/n} \right) \in G,\ m+n = d.
}
From   Schmidt's  \cite{Sc2} proof of the winning property of the set of badly approximable matrices  Dani \cite{dani}  deduced that the set 
$${\mathcal B}(F) := \{x\in  X: 
Fx \text{ is bounded}\}
$$
 is \textit{thick}; that is,  the intersection of this set with any non-empty open subset of $X$ has full \hd. In a follow-up paper  \cite{Da2} Dani was able to adapt Schmidt's argument to show  the thickness of ${\mathcal B}(F)$ for non-quasiunipotent flows on quotients $X$ of semisimple   Lie groups of
$\R$-rank $1$. (We recall that a one-parameter subsemigroup $F = \{g_t: t\ge 0\}$ of a Lie group $G$ is called \underbar{non-quasiunipotent} if 
\eq{genericF}{F = \{g_t: t\ge 0\}, \ g_t = \exp(t\bv),} and $\bv\in \fg := \Lie(G)$  is such that  $\ad \bv: \fg \to \fg$  has at least one eigenvalue with non-zero real part.)
 Then in 1996 the first-named author and Margulis \cite{KM} came up with an alternative approach utilizing mixing properties of the $g_t$-action on $X$. To state their results, recall that the \underbar{expanding horospherical subgroup relative to} 
 $g_1$, is defined as 
 $$U_+(g_1) := \{u\in G : g_{-t}ug_t\to e\text{ as }t\to +\infty\};$$
equivalently, as the exponential image of the sum of all the generalized eigenspaces of $\ad \bv$ corresponding to eigenvalues with positive real part.
{(Here and hereafter $e$ will stand for the identity element of Lie groups.)}\smallskip
 
The following is a special case of the main result\footnote{The paper \cite{KM} is written in a bigger generality than the set-up of Theorem \ref{km}; in particular for any Lie group $G$ and any lattice $\Gamma$ it gives a condition on a one-parameter subsemigroup $F$ of $G$ equivalent to the thickness of ${\mathcal B}(F)$. However the general case is reduced there to  $\Gamma$ being  an irreducible
lattice in  a connected semisimple Lie
group  $G$ without compact factors. 
} 
of  \cite{KM}:

\begin{theorem} \label{km} Let $G$ be a connected semisimple real Lie
group  without compact factors,  $\Gamma$ an irreducible
lattice in $G$, $X = \ggm$, and 
$F = \{g_t: t\ge 0\}$ a non-quasiunipotent subsemigroup of
$G$.   
Then 
\begin{itemize}
\item[\rm (a)] 
the set ${\mathcal B}(F)$ is thick;
\item[\rm (b)] let $U = U_+(g_1)$, then for any $x\in X$ the set 
$\big\{u\in 
U:  
ux \in {\mathcal B}(F)\big\}$ 
 is thick in $U$.
\end{itemize}
\end{theorem}
We remark that the $F$-action on $X$ is ergodic and moreover mixing, which implies that the sets 
${\mathcal B}(F)$ and 
$\big\{u\in 
U: 
ux \in {\mathcal B}(F)\big\}$
are Haar null {if $X$ is not compact}. 
\smallskip

The present work grew out of the following questions: what if a ray $\{\exp(t\bv): t\ge 0\}$ is replaced by an arbitrary subset $F$ of $G$? Under what conditions on $F$ can the set ${\mathcal B}(F)$ 
be shown to be thick? For example,  $F$ could be a countable set or a smooth curve. Also, if $U$ is a horospherical subgroup (that is, expanding horospherical  relative to some $g\in G$), under what conditions for any $x\in X$ there exists a thick  set
of $u\in
U$ such that $  
Fux$ is bounded?
Note that whenever $F$ is unbounded, {$X$ is not compact and the $G$-action on $X$ is mixing}, it follows 
that $Fx$ is dense in $X$ for Haar-almost every $x\in X$. {(Otherwise there would exist a non-empty open $B \subset X$, a subset $C$ of $X$ of positive measure   and a sequence $f_k$ tending to infinity in $G$ such that $f_kx\notin B$ for any $x\in C$ and all $k$, contradicting mixing.)}  
 Hence {under the above assumptions the set} ${\mathcal B}(F)$ is Haar null.  Likewise, if one can find a sequence $f_k\in F$ such that the $f_k$-translates of $Ux$ get equidistributed in $X$, it follows that $Fhx$ is dense in $X$ for Haar-a.e.\ $h\in U$, hence the set $\big\{u\in 
U:  
ux \in {\mathcal B}(F)\big\}$ has Haar measure zero.

\smallskip
In order to state our main result we need the following general 
\begin{definition}\label{qr} \rm Let $V$ be a finite-dimensional normed vector space, 
and let $\fc\subsetneq V$ be a 
convex 
cone. 
\begin{itemize}
\item Say that a sequence {$\mathscr{V} = ({\pmb v}_k)_{k\in\N}$} of elements of $\fc$ is a \underbar{quasi-ray in} $\fc$ if  $\bv_{k} - \bv_{k-1} \in \overline{\fc}$ 
for any $k \in\N$ {(here we put $\bv_0 = 0$)}, and the set {of differences} $\{\bv_{k} - \bv_{k-1}\}$ is 
precompact in $V\nz$ (in other words, the norms of $\bv_{k} - \bv_{k-1}$ are bounded between two positive constants).
 
\item  If $G$ is a Lie group 
and  $V = \Lie G$, say that $F\subset G$ is a  \underbar{quasi-ray relative to} $\fc$ if there exists a quasi-ray {$\mathscr{V}$} such that $F$ lies in 
a bounded thickening of $\exp({\mathscr{V}})$. In other words, $$\sup\limits_{g \in F} \inf_{k \in \N} \dist(g, \exp \bv_k) < \infty.$$ 
(Here "$\dist$" refers to any right-invariant Riemannian metric on $G$.)
\end{itemize}
\end{definition}

Clearly any one-parameter semigroup of $G$ of the form 
\equ{genericF} is a quasi-ray relative to 
any open cone containing $\bv$. But 
the notion of a quasi-ray is much less restrictive.  {Let us illustrate it by the following simple} 

\begin{example}
    Let $G = \SL_3(\R)$, $ \fg = \mathfrak{sl}_3(\R) = \Lie G$, {${\pmb w}_1  = \diag(1,2,-3)$ and ${\pmb w}_2 = \diag(2, 1, -3)$}. Let {$( \varepsilon_k )_{k \in \N }$} be an arbitrary sequence {with} $\varepsilon_k \in \{1,2\}$ for any $k$, {and} let $\bv_k: = \sum\limits_{i=1}^k {\pmb w}_{\varepsilon_i}$. Then {$\mathscr{V} = (\bv_k)$} is a quasi-ray in the cone 
    $$
    \fc: = \{ {\diag(y_1, y_2, y): {y_1, y_2 > 0, \  y < 0, \ y_1 + y_2 + y = 0}} \}.
    $$
    {Clearly one can choose $(\varepsilon_k )$ so that {$\mathscr{V}$} is not contained in a bounded thickening of a ray in $\fc$.}
\end{example}

Now we are ready to state the main result of the paper, modulo some more definitions that will be given in the next section. {In what follows an element $b$ of $G$ will be called $\Ad$-diagonalizable if $\Ad(b)$ is diagonalizable over $\R$.}

\begin{theorem} \label{main} Let $G$ 
be a connected semisimple real Lie
group  without compact factors, {and let $U$ be a nontrivial subgroup of $G$ that is expanding horospherical relative to some $\Ad$-diagonalizable $b\in G$.  Then there exists a
connected $\Ad$-diagonalizable subgroup $A\ni b$} of $G$
and a proper convex open cone ${{\fs}_\fu^+}$ in $\fa = \Lie A$ (here we let $\fu = \Lie U$)
such that the following holds:
for any
 irreducible
lattice $\Gamma$ in $G$, any cone $\fs$ in $\fa$ 
{with} $\overline{\fs} \subset {{\fs}_\fu^+} \cup \{ 0 \}$, any quasi-ray $F$ relative to $\fs$ and any  $x\in X = \ggm$, the set \eq{conclU}{\{u\in 
U:  
ux \in {\mathcal B}(F)\}}
is thick. Consequently, the set ${\mathcal B}(F)$ is thick as well. 
\end{theorem}

The cone ${{\fs}_\fu^+}$ will be called the \textit{super-expanding cone {with respect} to} {$\fu$} and will be defined in \S\ref{def_cones_section}. {Note that the subgroup $A$ in the above theorem is not defined uniquely (see the beginning of \S\ref{def_cones_section} for an explanation of how $A$ can be chosen),  and the cone ${{\fs}_\fu^+}\subset \fa = \Lie(A)$ depends on the choice of $A$.} It is contained in the \textit{expanding cone} defined by Shi in \cite{Shi}. Roughly speaking, it has the following properties: {translates of $U$-orbits on finite volume \hs s $\ggm$ by elements of the form $\exp(t\bv)$ for $\bv\in {{\fs}_\fu^+}$ get  equidistributed in $\ggm$ and at the same time locally expanded in all directions}. More discussion and examples will follow in {the Appendix}.

\subsection{Dynamics  on the space of lattices and \da} \label{intro2} To recall  the basics of simultaneous \da, we will use the notation $\mr$ for the set of $m \times n$ matrices with coefficients from $\R$.
For $k \in \N$, we will denote by $| \cdot |$  the supremum norm on $\R^k$, by $\| \cdot \|$  the distance to the nearest integer vector; that is, 
$$
\| \bx \| = \min\limits_{\p \in \Z^k} |\bx - \p | \quad\text{for }\bx\in \R^k,
$$
and by $1_k$   the $k\times k$ identity matrix. 
 Dirichlet's theorem in Diophantine approximation states: for any $\Theta \in \mr$ and any   $T \ge 1$ the system of inequalities 
\begin{equation}\label{Dirichlet_f1} 
\begin{cases}
    \|\Theta {\bf q}\|^m &\leq f_1(T) := \frac{1}{T} \\
    |{\bf q}|^n &\leq T
\end{cases}
\end{equation}
has a nonzero  solution $\vq \in \mathbb{Z}^n$. 

More generally, one can replace the approximating function $f_1$ 
by another non-increasing 
function $f:[1,\infty)\to \mathbb{R}_{>0}$ {tending to $0$ at infinity (hereafter referred to as an \underbar{approximation function})}. A matrix $\Theta \in \mr$ is called  \underbar{$f$-approximable} if the system of inequalities 
\begin{equation}\label{Dirichlet_hom}
    \begin{cases}
    \|\Theta {\bf q}\|^m &\leq f(T) \\
    |{\bf q}|^n &\leq T
\end{cases}
    \end{equation}
    has a nonzero  solution $\vq \in \mathbb{Z}^n$ for an unbounded set of $T$;
thus it follows from Dirichlet's theorem that all $\Theta \in \mr$ are $f_1$-approximable.
This motivates the definition of  
badly approximable matrices as the ones for which $f_1$-approximability can only be improved 
up to some multiplicative constant:
 $\Theta \in \mr$ is called \underbar{badly approximable} if it is not ${\varepsilon}f_1$-approximable for some ${\varepsilon} > 0$.
We will use the notation \ba\ for the set of badly approximable  $m \times n$ matrices. 
This set has been extensively studied: namely, it is known that it has zero Lebesgue measure but is thick, the latter being established in the aforementioned paper \cite{Sc2} by Schmidt. 
To describe this set using dynamics, one needs to specialize to $G$, $\Gamma$ and $X$ as in 
\equ{sld}; that is, let $X = X_d$ be the set of unimodular lattices in $\R^d$. This will be a standing assumption until the end of this subsection.
With $F = \{g_t: t\ge 0\}$  as in \equ{defgt},  the Dani  Correspondence \cite[Theorem 2.26]{dani} states that 
\eq{danicorr}{\Theta \in{\bf BA}_{m,n} \iff u_{\Theta} \Z^d\in {\mathcal B}(F),} where \eq{defua}{u_{\Theta}: = \begin{pmatrix}
    1_m & \Theta \\ & 1_n
\end{pmatrix}.} Also it is easy to see that 
\eq{ua}{U_+(g_1) =  \{
u_{\Theta}: \Theta \in \mr \},}
thus the thickness of \ba\ gets to be a special case of Theorem \ref{km}(b).
\smallskip

Now let us consider actions by subsets of $G$ that are not necessarily semigroups. The main \di\ application of Theorem \ref{main} is to studying 
{an analogue of the set \ba}
 in a more general set-up of \textit{approximations with weight functions}. We follow the notation introduced in \cite{N25}.
Let 
$$
\alpha_i:(0,1] \rightarrow (0,1] , \  i = 1, \ldots, m \quad\text{and} \quad  \beta_j:[1,\infty) \rightarrow [1,\infty),\  j = 1, \ldots, n
$$
be  monotonically increasing bijections. 
Take two tuples 
$$
\alf: = ( \alpha_1, \ldots, \alpha_m) \,\,\,\, \text{and} \,\,\,\, \betf: = (\beta_1, \ldots, \beta_n )
$$
as   above and    an approximation function   $f$. Generalizing the system \eqref{Dirichlet_hom}, one can say that 
$\Theta \in \mr$ is  $(f,\alf, \betf)$-approximable if the system of inequalities \begin{equation}\label{Dirichlet_hom_fcns}
    \begin{cases}
    \|\Theta_i {\bf q}\|  &\leq \alpha_i\big(f(T)\big),\ i = 1,\dots,m \\
    |q_j| &\leq \beta_j(T),\quad \ \ \, j = 1,\dots,n
\end{cases}
    \end{equation}
    has a nonzero  solution $\vq \in \mathbb{Z}^n$ for an unbounded set of $T$.
    Here $\Theta_1,\dots,\Theta_m$ are the rows of $\Theta$.
We will use the notation $\wab(f)$ for the set of $\fab$-approximable $m \times n$  matrices.

{At this point} one can readily see that there is room for some simplifications. Indeed, let $\varphi$ and $\psi$ be continuous increasing self-bijections of $(0,1]$ and $[1,\infty)$ respectively. Then, using the change of variables $T = \psi(T')$, one can see that $$\wab(f) = {\bf W}_{\alf\circ \varphi, \betf\circ \psi}(\varphi^{-1}\circ f\circ \psi).$$ 
Thus without loss of generality one can replace functions $\alpha_i$ and $\beta_j$ with $\alpha_i\circ \varphi$ and $\beta_j\circ \psi$ for some conveniently chosen $\varphi$ and $\psi$. In fact, an excellent choice for those turns out to be $\varphi = \left( \prod\limits_{i=1}^m \alpha_i \right)^{-1}$ and $\psi = \left( \prod\limits_{j=1}^n \beta_j \right)^{-1}.$ Such a replacement implies 
 \eq{normalizatiion}{\prod\limits_{i=1}^m \alpha_i (1/T) = 1/T\quad \text{and}\quad\prod\limits_{j=1}^n \beta_j(T) = T,}
which will be our standing assumption for the rest of the paper. Assuming  \equ{normalizatiion}, one can apply Minkowski's Convex Body Theorem {\cite[Ch. III, 
Theorem II]{Cas}}
and conclude that the system \eqref{Dirichlet_hom_fcns} with $f=f_1$ has a nonzero integer solution for any $T \ge 1$. In particular, any \amr\ is $(f_1,\alf, \betf)$-approximable. 
This  motivates a natural generalization of the definition of \ba:  a matrix \amr\ is called $(\alf, \betf)$-badly approximable if it is not $({\varepsilon}f_1,\alf, \betf)$-approximable  for some ${\varepsilon} > 0$.
We will use the notation \baab\ for the set of $m \times n$ $(\alf, \betf)$-badly approximable matrices.

Clearly the definition of \ba\  is a special case, with 
$$\alpha_i(x) = x^{{1}/{m}}, \ i = 1,\dots,m\quad\text{and}\quad \beta_j(T) = T^{{1}/{n}}, \ j = 1, \ldots, n.$$  
Recently  in \cite{N25} the second-named author extended the Dani correspondence \equ{danicorr} 
assuming that functions $\alpha_i$ and $ \beta_j$ satisfy some mild regularity condition referred to as quasimultiplicativity (see definitions and discussion in \S\ref{qmd}). 
Namely it was shown {in \cite{N25}} that, assuming that $\alpha_i$ and $ \beta_j$  are quasimultiplicative, $\Theta\in {\bf BA}_{\alf, \betf}$ if and only if the trajectory of $u_\Theta\Z^d$ under the action of the curve  
\eq{curveinT}{
\begin{aligned}&F^{\alf, \betf}:=\big\{g(t): t\ge 0\big\},\text{ where }\\g(t) := \diag &\left( \frac1{\alpha_1(e^{-t})}, \ldots, \frac1{\alpha_m(e^{-t})}, \frac1{\beta_1(e^{t})}, \ldots, \frac1{\beta_n(e^{t})} \right),
\end{aligned}}
is bounded in $X_d$.
(See Proposition \ref{Dani_corr} 
for a restatement.) Note that $F^{\alf, \betf}\subset G$ in view of \equ{normalizatiion}. 
 In \S\ref{qmd} we  show   that the quasimultiplicativity of $\alf$ and $\betf$ 
implies that $F^{\alf, \betf}$ is a quasi-ray relative to some cone $\fs$ {whose} closure 
is contained in the cone
\eq{aplus}{ \left\{\diag(a_1, \ldots, a_m,-b_1, \ldots, -b_n): a_i,b_j > 0, \ \sum\limits_{i=1}^m a_i = \sum\limits_{j=1}^n b_j \right\}.}
{It is easy to show, see  Example \ref{exampleua}, 
that 
the cone \equ{aplus}} is precisely the super-expanding cone ${\fs_\fu^+}$ for 
${\fu = \Lie\big(U_+(F)\big)}$ as in \equ{ua} and $F$ as in \equ{defgt}. 
This, together with Theorem~\ref{main}, implies

\begin{theorem}\label{badisfull}
    Suppose $\alpha_i$, $i = 1,\dots,m$ and $\beta_j$, $j = 1, \ldots, n$, are quasimultiplicative functions {satisfying \equ{normalizatiion}}. Then
    \baab\ is thick in $\mr$.
\end{theorem}

Note that a  well-studied special case of the above set-up is given by 
$$\alpha_i(x) = x^{a_i}, \ i = 1,\dots,m,\quad\text{and}\quad \beta_j(T) = T^{b_i}, \ j = 1, \ldots, n,$$
where
$(a_1, \ldots, a_m)$ and $(b_1, \ldots, b_n)$ are an $m$-tuple and an $n$-tuple of positive real numbers
satisfying 
$\sum\limits_{i=1}^m a_i = \sum\limits_{j=1}^n b_j = 1$.
This
is the
set-up of \textit{Diophantine approximation with weights}. 
In this case $F^{\alf, \betf}$ is a semigroup of the form \equ{genericF} with $\bv = \diag(a_1, \ldots, a_m,-b_1, \ldots, -b_n)\in\fs_{\fu}^+$,  the correspondence of Proposition \ref{Dani_corr} in \S\ref{qmd} has been previously worked out by the first-named author in \cite{K98}, and the thickness of \baab\ was proved by Pollington and Velani in the case $\min(m,n) = 1$ and in \cite{KW08} in the general case; see also \cite{BNY} for the hyperplane absolute winning property of this set in the case $\min(m,n) = 1$. However the set-up of arbitrary quasimultiplicative weight functions is much more general. {For a nontrivial example, see \cite[\S 1.7]{N25}.}

\subsection{{The structure of the paper}}\label{struct}
The main goal of \S\ref{def_cones_section} is to revisit the notion of expanding cones with respect  to horospherical subgroups $U$. We need two types of such cones: the one defined in \cite{Shi} ($A^+_U$) that we call algebraically expanding, and the one that is only implicitly present in \cite{Shi}, that we refer to as the geometrically expanding  cone ($E^+_U$). Our focus, however, 
is on the intersection of these cones that we call the super-expanding cone ($S^+_U$). It is prominently featured in the equidistribution estimate that we use for our proof. We also discuss the notion of distance from walls of these cones and use it to simplify the statement of Shi's equidistribution theorem.

\S\ref{mainproof} is devoted to the proof of Theorem \ref{main} via a construction of Cantor-type sets consisting of points with bounded trajectories.  The strategy of the proof follows the lines of \cite{KM}, but certain modifications are needed in order to adapt the argument to actions of sequences instead of semigroups. 
In \S\ref{qmd} we derive the   \di\ consequences of Theorem \ref{main}: full \hd\ of the set of badly approximable systems of linear forms with respect to quasi-multiplicative weight functions (Theorem \ref{badisfull}). This involves a discussion of the notion of quasi-multiplicativity and a generalization of the classical Dani correspondence, that is, relating  bad approximability of matrices to the boundedness of the corresponding  quasi-ray trajectories in the space of lattices. 

At the end of the paper we return to the topic of expanding cones. In the Appendix, which is not needed for the main results of the paper and is included simply for illustrative purposes, we explicitly describe all the algebraically and geometrically expanding  cones inside the diagonal subgroup of $\SL_d(\R)$, and prove necessary and sufficient conditions for these cones lying inside one another.

\subsection{Acknowledgments}
The authors are grateful to Victor Beresnevich, Nikolay Moshchevitin and Ronggang Shi for stimulating discussions.  

\section{Expanding and super-expanding cones}\label{def_cones_section}

    {In this and the next section we keep the assumptions of Theorem \ref{main}. That is,  we let $G$ 
be a connected semisimple Lie
group  without compact factors,  $\Gamma$ an irreducible lattice in $G$, {$X = \ggm$},  
{and $U$  a nontrivial subgroup of $G$ that is expanding horospherical relative to some $\Ad$-diagonalizable element of $G$. Let $H$ be the product of all the simple factors of $G$ such that the projection of $U$ to each of them
is nontrivial. Then $U = U_+(b)$ for some $\Ad$-diagonalizable   $b\in H$.
Equivalently {(see 
discussion on page 302 of \cite{DR80})}, $U$ is 
the unipotent radical of a  parabolic subgroup $P\ni b$ of $H$ that is absolutely proper, that is, the projection of $P$ to each simple factor of $H$ is not surjective. We let $A\ni b$ be a maximal connected $\Ad$-diagonalizable subgroup
of $P$. Note that $b,P$ and $A$ are in general not defined uniquely given $U$ as above.}}

We let $\fg = \Lie(G)$, $\fh = \Lie(H)$, $\fa = \Lie(A)$ 
and $\fu = \Lie(U)$. {Let $\langle\cdot,\cdot\rangle$ and $|\cdot|$ be the inner product and the norm on $\fa$ given by the Killing form  on $\fg$; we will use the same notation for the induced inner product and norm on $\fa^* $, the dual space of $\fa$. 
{Extending $\langle\cdot,\cdot\rangle$ to a positive definite form on $\fg$, one  defines a}  Riemannian structure on $\fg$, the corresponding right-invariant Riemannian metric "$\dist$" on $G$, {and the induced metric on $X$.}
Note that this way the exponential map $\exp: \fa\to A$ becomes an isometry.}

    {Let $\Phi$ be the set of restricted roots of $\fh$ relative to $\fa$, and let $\Phi_\fu : = \{ \lambda \in \Phi: \fh_{\lambda} \subseteq \fu \}$, where 
    $$
    \fh_{\lambda} = \left\{ {\pmb y} \in \fh: [{\bv},{\pmb y}] = \lambda({\bv}){\pmb y} \ \text{for any} \ {\bv} \in \fa \right\}.
    $$}
    Then we have 
    \eq{rootdecomp}{
\mathfrak u=\bigoplus_{\lambda\in\Phi_{\fu}}\fh_\lambda}
    {(see \cite[(7.77a) and (7.77b)]{Knapp}). }
    We consider two cones in $A$ (and cones in $\fa$ corresponding to them) that can be defined for any such $U$.
    
    \begin{itemize}

    \item 
    For any $\lambda \in \fa^*$ 
    define ${\bv}_{\lambda} \in \fa$ via the relation 
\begin{equation}\label{s_alpha_def}
        {\langle {\bv}_{\lambda}, {\bv} \rangle = \lambda({\bv}) \,\,\,\, \text{for every} \,\, {\bv} \in \fa,\text{ or, equivalently, } \mu({\bv}_{\lambda}) = \langle \mu, \lambda \rangle \ \text{for any} \    \mu \in \fa^*,}
        \end{equation}
        and then define the cone $\fa_{\fu}^+$ as
\begin{equation}\label{au_big}
            \fa_{\fu}^+: = \left\{ \sum\limits_{\lambda \in \Phi_\fu } t_{\lambda}{\bv}_{\lambda}: t_{\lambda} > 0 \right\}.
        \end{equation}
        We let $A_U^+: = \exp(\fa_{\fu}^+)$. {These cones were defined in \cite{Shi}, where $A_U^+$ was called the \textit{expanding cone with respect to} $U$, 
        and it was proved that $a\in A$ is in $A_U^+$ if and only if
        for every nontrivial irreducible representation $(\rho,V)$ of $H$, the space of  $\rho(U)$-invariant vectors in $V$ is contained in the span of  eigenvectors of $\rho(a)$ with eigenvalues of modulus bigger than $1$.

        In view of this algebraic restatement we will refer to the cones $\fa_{\fu}^+$ and $A_U^+$ as  \underbar{algebraically expanding} cones.}
        \smallskip 
        
        \item We also define the  \underbar{geometrically expanding} cone  $E_U^+$ 
        to be the set of elements in $A$ expanding $U$ by conjugations; that is, 
        \begin{equation}\label{EU}
            E_U^+: = \{ a \in A: a^{-k} u a^k \rightarrow 
            {e}\text{ as } k\to\infty\text{ for any} \,\, u \in U \} = \{a\in A: U\subset U_+(a)\}.
        \end{equation}
        It is clear that $E_U^+$ is a cone in $A$. 
        On {the} Lie algebra level it corresponds to the cone
        \begin{equation}\label{eu}
            \fe_{\fu}^+: = \{ {\bv} \in \fa: \lambda({\bv}) > 0 \,\, \text{for any} \,\, \lambda \in \Phi_\fu  \}.
        \end{equation}

    \end{itemize}

In the present paper we need a cone with both {algebraic and geometric} expanding properties.
    
    \begin{definition}\label{superexp}
        We define the \underbar{super-expanding cone} with respect to $U$ in $A$ by $S_U^+: = E_U^+ \cap A_U^+$. On the Lie algebra level, we analogously define the {super-expanding cone} with respect to $\fu$ in $\fa$ by $
        \fs_\fu^+: = \fe_\fu^+ \cap \fa_\fu^+$.
    \end{definition}
    It follows from the discussion above that $S_U^+ = \exp(\fs_{\fu}^+)$. 

    \begin{example}\label{exampleua} {Let us illustrate these notions in the case that is important for \di\ applications, that is, $H = G = \SL_d(\R)$ for $d\ge 2$, $U$ of the form \equ{ua} for some $m,n\in \N$ with $m+n=d$, and $A$ the full diagonal subgroup of $H$. In this case $U$ is a minimal {(equivalently for $\SL_d(\R)$, abelian)} horospherical subgroup of $H$.
    This example is thoroughly described in Shi's paper, see \cite[Example 1.1 and \S2.3]{Shi}, where it is  shown that  the algebraically expanding cone $A_U^+$ is of the form
    \begin{equation}\label{ex22_cone_a}\left\{\diag(e^{a_1}, \ldots, e^{a_m},e^{-b_1}, \ldots, e^{-b_n}): a_i,b_j > 0, \ \sum\limits_{i=1}^m a_i = \sum\limits_{j=1}^n b_j \right\},\end{equation}
    that is, the exponential image of the cone \equ{aplus}. On the other hand, it is clear that the  geometrically expanding cone {with respect to $U$} is given by
    $$E_U^+ =  \left\{\diag(e^{a_1}, \ldots, e^{a_m},e^{-b_1}, \ldots, e^{-b_n}): a_i+b_j > 0, \ \sum\limits_{i=1}^m a_i = \sum\limits_{j=1}^n b_j \right\}.$$
    Hence in this case $S_U^+ = A_U^+ \subset E_U^+$.} {{We note that this is the only possible example of an upper-triangular horospherical subgroup for which the above  inclusion holds (see Proposition \ref{a_in_e_sl}). In the case of an arbitrary semisimple $H$ the inclusion $A_U^+ \subset E_U^+$ is true for any abelian horospherical subgroup $U$, but not necessarily true for any minimal one; see \S \ref{abelian} for details and proof.}}
    \end{example}

\begin{example}\label{examplemax} For comparison let us consider the case when $U$ is maximal horospherical in  $H =  \SL_d(\R)$, which essentially (modulo conjugation) means that $U$ is the upper-triangular unipotent subgroup. Then it is easy to show {(see \eqref{e_sld_min} for an expression for $E_U^+$, where $U$ is an  arbitrary upper-triangular horospherical  subgroup of $\SL_d(\R)$)} that $E_U^+$ is a Weyl chamber in $A$, that is,  $$E_U^+ =  \left\{\diag(e^{y_1}, \ldots, e^{y_d}): y_i > y_j \text{ for } i<j,\text{ and } \sum\limits_{i=1}^d y_i = 0 \right\}.$$ Also, {as shown in Lemma \ref{a_plus_in_coord},} the algebraically expanding cone $A_U^+$ is of the form
    {
\begin{equation}\label{ex23_cone_a}
    \left\{\diag(e^{y_1}, \ldots, e^{y_d}):  \sum\limits_{i=1}^k y_i > 0 \text{ for } 1 \leq k \leq d-1,\text{ and } \sum\limits_{i=1}^d y_i = 0 \right\}
    \end{equation}
       Hence in this case $S_U^+ = E_U^+ \subset A_U^+$. There are other possible cases when the inclusion $E_U^+ \subset A_U^+$ holds; see Proposition \ref{eina_prop} for a criterion.
        }
    \end{example}

     \begin{remark}\label{shah} \rm It follows from \cite[Lemma 5.2]{s96} that the cone $\{a\in A: U_+(a) \subset U\}$ is a subset of the expanding cone $A_U^+$. This implies that the smaller (but still non-empty) cone $\{a\in A: U_+(a) = U\}$ is contained in the super-expanding cone $S_U^+$. In particular $S_U^+\ne\varnothing$. {Also it easily follows {from the definition of $E_U^+$ and the representation-theoretic description of $A_U^+$ discussed above} 
     that for two nontrivial horospherical subgroups $U_1\subset U_2$ it holds that $A_{U_1}^+\subset A_{U_2}^+$ and $E_{U_1}^+\supset E_{U_2}^+$. However the two super-expanding cones $S_{U_1}^+$ and $ S_{U_2}^+$ do not have to be contained in one another, as illustrated by Examples \ref{exampleua} and \ref{examplemax}.} 
 \end{remark}

    Next we describe the main application of expanding and super-expanding cones coming from Shi's paper \cite{Shi}: equidistribution results for actions on \hs s. Let us start by recalling a certain way to measure the distance of elements of cones from walls that was employed in \cite{Shi}. Suppose {$C\subsetneq A$ is a convex open cone  (here we endow $A$ with a vector space structure by means of the exponential map $\fa\to A$),} and let  $( b_t)_{t\in \R}$ be
 a one-parameter   subgroup  of $A$ such that   the positive ray $\{ b_t : t> 0\}$ is contained in $C$. 
Then 
 for $a\in A$   one can define 
\begin{align}\label{Cfloor} 
\lfloor a\rfloor_{C,b}  := \sup \{t\ge 0:  ab_{-t}\in C\},
\end{align}
    where $b$ stands for $b_1$. It is easy to see that, for any $(b_t)$ as above, $\lfloor a \rfloor_{C,b} $ is positive if and only if $a \in C$, and is zero otherwise (modulo the convention $\sup \varnothing = 0$).
    
    Now suppose that $b$ is such that the following two conditions hold:  \eq{nontriv}{\text{the projection 
of $b$ to each simple factor of {$H$} is not trivial,}} and \eq{contain}{\text{the expanding horospherical subgroup ${U'} = U_+(b)$ is contained in }U.} Then, as was mentioned in Remark \ref{shah},  the positive ray $\{ b_t : t> 0\}$ is contained in $A_U^+$.
   In \cite[(1.9)]{Shi}  the following expression was defined:
  \begin{align}\label{shifloor}
\lfloor a\rfloor:= \sup \{t\ge 0:  ab_{-t}\in A_U^+, \, {U'}  \subset  U_+({ab_{-t}})\}.  
\end{align}
    Combining \eqref{Cfloor} with \eqref{EU}, one can immediately see that     $\lfloor a\rfloor$ as above is equal to 
 $\lfloor a\rfloor_{A_U^+\cap {E_{U'}^+},b}$. 
 
 \smallskip
 Let us now state a special case of \cite[Theorem 1.5]{Shi}.  For that we  let $\mu_G$ be the Haar measure on $G$ which locally projects to $\mu_X$, the probability Haar measure on $X$. 
 Also  choose a Haar measure $\mu_U$ on $U$. 

\begin{theorem}\label{thm:effective}
{Let  
$G, X, U, H$ and $A$ be as above, }and let   
$b\in A$ be such that \equ{nontriv} and \equ{contain} hold.
Then there exists {$\zeta=\zeta( H, X,   b )>0$} with the following property:
for   any compact subset ${Q}$ of $ X $, any 
$\varphi \in C^\infty_c(U)$ and any $\psi \in  
C^\infty_c(X)$ 
  there exists $M=M({Q}, \varphi, \psi)>0$
  such that for any $x \in {Q}$ and $a\in A$ one has
  \eq{equidistr}{
    \left| \int\limits_U \varphi(u) \psi(aux) \,d \mu_U - \int\limits_U \varphi \,d \mu_U \int\limits_X \psi \,d \mu_X  \right| \leq Me^{- \zeta \lfloor a \rfloor},
 }
    where $\lfloor \cdot \rfloor = \lfloor \cdot\rfloor_{A_U^+\cap {E_{U'}^+},b}$ is as defined in \eqref{shifloor}.
\end{theorem}

We note that \cite[Theorem 1.5]{Shi} was stated under the assumption that the action of {$ H $} on $X$ has a spectral gap. The latter is known to hold for quotients of    connected semisimple Lie
group  without compact factors by  irreducible lattices; {see a discussion 
in \cite[p.\ 285]{KS09} for details and proof.}

\smallskip

{Next let us observe}  that $\lfloor \cdot \rfloor$ {in the above theorem} can be chosen in a $b$-independent way. {Denote by $\fa_1^*$ the unit sphere in $\fa^*$.}
Let $C$ be a 
cone in $A$, and {put} $\fc = {\log}(C)$. {Define} $$\mathcal{F}(\fc): = \{ \lambda \in \fa_1^*: \lambda({\bv}) > 0 \ \forall \,{\bv} \in \fc \}.$$ 
{It is known} (see \cite[Lemma 2.17]{AT}) 
that, {whenever $\fc$ is non-empty, open and convex, one has}
\begin{equation}\label{e_c_defines_cone}
    \fc = \{ {\bv} \in \fa: \lambda({\bv}) > 0 \ \forall\, \lambda \in \mathcal{F}(\fc) \}.
\end{equation}
{This} allows us to prove 

\begin{lemma}\label{b_doesnt_matter}
    Let {$C\subsetneq A$} be a {non-empty} convex open cone. Let $a, b \in C$, {and put} $\fc = \log(C)$ {and}  $\bv = \log(a)$. Then 
    $$
    \lfloor a \rfloor_{C,b} \asymp_b \lfloor a \rfloor_{C}: = \dist(a, \partial C) = {\dist(\bv, \partial \fc)} = \min\limits_{\lambda \in \mathcal{F}(\fc)} \lambda({\bv}),
    $$
    where the implied constant depends on $b$.
\end{lemma}

\begin{proof} {The equality $\dist(a, \partial C) = \dist(\bv, \partial \fc)$ is clear since we chose the metric so that the exponential map $\fa\to A$ is an isometry.} 
Let us first show that $\dist(a, \partial C) = \min_{\lambda \in \mathcal{F}(\fc)} \lambda({\bv})$. 
Let ${\pmb w} \in \partial \fc$ be 
such that $\dist(\bv, \partial \fc) = \dist(\bv, {\pmb w})$. Since ${\pmb w} \in \partial \fc$, {by \eqref{e_c_defines_cone} }
{there exists $\lambda_0 \in \mathcal{F}(\fc)$ such that $\lambda_0({\pmb w})$ is non-positive, and by the continuity of $\lambda_0$ we conclude that $\lambda_0({\pmb w}) = 0$}. Thus $\dist(\bv, {\pmb w}) \geq \dist(\bv, \Ker \lambda_0) = \lambda_0(\bv)$. 
{Conversely, for any ${\bv} \in \fc$ and $\lambda \in \mathcal{F}(\fc)$ one has $\Ker \lambda \cap \fc = \varnothing$ and $\lambda({\bv}) = \dist(\bv, \Ker \lambda)$, therefore, 
\begin{equation}\label{ineq_lambda_dist}
    \lambda({\bv}) \geq \dist(\bv, \partial \fc) \,\,\,\, \text{for any} \ {\bv} \in \fc, \ \lambda \in \mathcal{F}(\fc),
\end{equation}}
{hence} $\lambda_0(\bv) = \min_{\lambda \in \mathcal{F}(\fc)} \lambda({\bv})$, and we conclude that
\begin{equation}\label{eq_through_lambda0}
    {\dist(\bv, \partial \fc)} = \lambda_0(\bv) = \min\limits_{\lambda \in \mathcal{F}(\fc)} \lambda({\bv}).
\end{equation}

{Next}
we prove the '$\asymp_b$' part. {Let ${\pmb y} = \log b$.} By definition, 
\begin{equation}\label{boundary_elt}
\bv - \lfloor a \rfloor_{C,b}{{\pmb y}} \in \partial \fc,
\end{equation}
and so 
$$
\dist(a, \partial C) = \dist(\bv, \partial \fc) \stackunder[1pt]{{}\leq{}}{\scriptstyle  \eqref{boundary_elt}} \dist(\bv, \bv - \lfloor a \rfloor_{C,b}{{\pmb y}}) = \lfloor a \rfloor_{C,b}|{{\pmb y}}|.
$$
Conversely, 
$$
\dist(a, \partial C) \stackunder[1pt]{{}={}}{\scriptstyle\eqref{eq_through_lambda0}} \lambda_0(\bv) =   \lambda_0(\bv - \lfloor a \rfloor_{C,b}{{\pmb y}}) + \lfloor a \rfloor_{C,b} \lambda_0({{\pmb y}}) 
$$
$$
\stackunder[1pt]{{}\text{}\geq{}}{\scriptstyle \substack{ \lambda_0(\bv - \lfloor a \rfloor_{C,b}{{\pmb y}}) \geq 0 \\ \text{by} \ \eqref{boundary_elt}}} \lfloor a \rfloor_{C,b} \lambda_0({{\pmb y}}) \stackunder[1pt]{{}\geq{}}{\scriptstyle  \eqref{ineq_lambda_dist}}  \lfloor a \rfloor_{C,b} \dist({{\pmb y}}, \partial \fc)  = \lfloor a \rfloor_{C,b} \dist(b, \partial C),
$$
{which finishes the proof.}
\end{proof}

{The above lemma shows that the statement of Theorem~\ref{thm:effective} can be simplified with the dependence on $b$ removed. Furthermore, in this paper we are interested in the super-expanding cone $S_U^+ = A_U^+\cap E_U^+$. As we saw in  Remark \ref{shah},   if, strengthening \equ{contain},   one chooses $b$ such that  $U = U_+(b)$, then   the positive ray $\{ b_t : t> 0\}$ will be contained in $S_U^+$. Moreover,  
 \equ{nontriv} will hold automatically since $P$ was {chosen to be absolutely proper in $H$}. 
In view of that, as well as of   Lemma \ref{b_doesnt_matter}, Theorem~\ref{thm:effective} implies}
 
\begin{corollary}\label{cor:effective}
{Let 
$G, X, U, H$ and $A$ be as above}.
Then there exists {$\zeta=\zeta( H, X, {U, A})>0$} with the following property:
for any compact subset ${Q}$ of $ X $, any 
$\varphi \in C^\infty_c(U)$ and any $\psi \in 
C^\infty_c(X)$ 
  there exists $M=M({Q}, \varphi, \psi)>0$
  such that \equ{equidistr}, with $\lfloor \cdot \rfloor = \lfloor \cdot\rfloor_{S_U^+}$, holds for any $x \in {Q}$ and $a\in {S_U^+}$.
\end{corollary}

The last
{part of the} section connects the quasi-ray property with {the notion of} "distance to walls" (cf.\ \cite[(1.8)]{KW08}) and allows us to apply Corollary \ref{cor:effective} in the set-up of Theorem \ref{main}. {It will be convenient to slightly abuse notation and for an element $\bv$ of a {convex} 
cone $\fc\subsetneq \fa$   denote}  
$$\lfloor \bv \rfloor_{\fc}: = \dist(\bv, \partial \fc).$$
Note that the function $\lfloor \cdot \rfloor_{\fc}$ has a concavity property whenever $\fc$ is convex: \eq{convexity}{\lfloor \bv_1 + \bv_2 \rfloor_{\fc} \ge \lfloor \bv_1 \rfloor_{\fc} + \lfloor \bv_2 \rfloor_{\fc}\text{ for all }\bv_1,\bv_2\in \fc.}
This can be easily shown using \eqref{eq_through_lambda0}.

{Now for a sequence {$\mathscr{V} = (\bv_k)_{k\in\N}$} of elements of   
$\fa$ such that the differences $\bv_k - \bv_{k-1}$ (here we again put $\bv_0 = 0$) lie in a    convex     cone $\fc\subsetneq \fa$, let us define 
$$
\rho_\fc(\mathscr{V}) := \inf\limits_{k \in \N}\lfloor\bv_k - \bv_{k-1}\rfloor_{\fc}
=\inf\limits_{\lambda \in \mathcal{F}(\fc), \,k \in \N}
            \lambda(\bv_k - \bv_{k-1}) ,
$$
where the equality above is given by Lemma  \ref{b_doesnt_matter}.
\begin{lemma}\label{from_s_to_sc}
        {For $\fc$ and $\mathscr{V}$ as above,} the following are equivalent:
        \begin{enumerate}
            [label= \rm (\roman*)]\item\label{stosc1} There exists {a convex} open cone $\fs$ in $\fa$ with $\overline{\fs} \subset \fc \cup \{ 0 \}$  such that {$\mathscr{V}$}
            is a quasi-ray in $\fs$.
            \item\label{stosc2} $\mathscr{V}$
            is a quasi-ray in $\fc$, and ${\rho_\fc(\mathscr{V})
            >0}$.
            \item\label{stosc3} The differences $\bv_k - \bv_{k-1}$ are uniformly bounded, 
            and 
            ${\rho_\fc(\mathscr{V})
            >0}$.
        \end{enumerate}
    \end{lemma}
    \begin{proof}
        To show {\ref{stosc1} $\Rightarrow$ \ref{stosc2},} let ${\fa_1}$ be the unit sphere in $\fa$ with respect to the norm $| \cdot |$. The set $\overline{\fs} \,\cap\, {\fa_1} \subset \fc$ is closed, and so is the set $\partial \fc$. Since these sets are disjoint, the distance between them is positive. Let ${\varepsilon}: = \inf\limits_{k \in \N} |\bv_k - \bv_{k-1}|$; {it is positive since $\mathscr{V}$ is a quasi-ray}, and clearly $\rho_\fc(\mathscr{V})\geq 
        {\varepsilon \dist(\overline{\fs} \cap {\fa_1}, \partial \fc)}
        $.
        
         The implication {\ref{stosc2} $\Rightarrow$ \ref{stosc3}} is obvious. And finally, 
           {let  \eq{defL}{L :=  \sup_{k\in\N}|\bv_k - \bv_{k-1}|;}
           then it is clear that the cone $$\fs = \left\{ \bv \in \fa: \lambda(\bv) > {\rho_\fc(\mathscr{V})} | \bv |{/ (2L)} \  \text{for all} \,\, \lambda \in \mathcal{F}(\fc)  \right\}$$ 
satisfies $\overline{\fs} \subset \fc \cup \{ 0 \}$, and for any $k\in N$ and $\lambda \in \mathcal{F}(\fc)$ one has
$$|\bv_k - \bv_{k-1}| \geq \lambda(\bv_k - \bv_{k-1}) \ge \rho_\fc(\mathscr{V}) \ge \rho_\fc(\mathscr{V}) |\bv_k - \bv_{k-1}|/L,
$$
hence $\bv_k - \bv_{k-1}\in \fs \subset \overline{\fs}$ and the norms of $\bv_k - \bv_{k-1}$ are uniformly bounded from below. Thus {$\mathscr{V}$}
            is a quasi-ray in $\fs$, which proves   the implication {\ref{stosc3} $\Rightarrow$ \ref{stosc1}}.}
    \end{proof}
\begin{corollary}\label{unbdd}Let $\mathscr{V}$ be a quasi-ray in a convex cone $\fc\subsetneq\fa$ with $\rho_\fc(\mathscr{V})>0$. Then  $\mathscr{V}$ is unbounded.
\end{corollary} 
\begin{proof}It follows from the above lemma and \equ{convexity}  that $|\bv_k|$ is not less than
\eq{lingrowth}{ 
\lfloor\bv_k\rfloor_{\fc} = \left\lfloor\sum_{i=1}^k(\bv_i - \bv_{i-1})\right\rfloor_{\fc} \ge \sum_{i=1}^k\left\lfloor\bv_i - \bv_{i-1}\right\rfloor_{\fc} \ge k  \rho_\fc(\mathscr{V}).
}
 \end{proof}
 Note that the assumption  $\rho_\fc(\mathscr{V})>0$ is necessary for the conclusion: indeed,  the sequence $ \big((1,1),(-1,1),(1,1), (-1,1),\dots\big)$  is a quasi-ray in $\{(x,y): y> 0\}\subset \R^2$.}
 \smallskip
 
{We close the section by describing a procedure of "thinning" a quasi-ray to increase the distance-from-walls parameter $\rho_\fc(\mathscr{V})$. 
\begin{proposition}\label{subray}Let $\mathscr{V}$ be a quasi-ray in a convex cone $\fc\subsetneq\fa$ with $\rho_\fc(\mathscr{V})>0$. Then  for any $t>0$ there exists   $\mathscr{V}'\subset \mathscr{V}$ with 
$\rho_\fc(\mathscr{V}')\ge t$   such that $\mathscr{V}$ is contained in a bounded thickening of $ \mathscr{V}'$.\end{proposition} 
\begin{proof} 
{For $ \mathscr{V} = (\bv_k)$, let  $L$ {be as in \equ{defL}, 
and put} $\rho = \rho_\fc(\mathscr{V})$.} 
Fix $t > 0$. We will construct $ \mathscr{V}' = (\bv_{k_i})_{i\in\N}$, inductively defining the sequence $({k}_i)$ of 
integers as follows: put $k_0 = 0$, and, assuming $k_{i-1}$ is defined, let $$k_i := \min\big\{k: |\bv_{k} - \bv_{k_{i-1}}| \geq   L t/\rho\big\};$$ it exists in view of Corollary \ref{unbdd}. 
Then $Lt/\rho \leq |\bv_{k_i}- \bv_{k_{i-1}}| \leq Lt/\rho + L$ for every $i\in\N$. It is also easy to see that 
$k_i - k_{i-1} \geq t/\rho$; 
therefore a computation similar to
\equ{lingrowth} yields that $\lfloor \bv_{k_i} - \bv_{k_{i-1}} \rfloor_{\fc} \geq  t$, i.e.\ $ \rho_\fc(\mathscr{V}') \geq  t$. 
Also by construction for any $k$ with $k_{i-1} < k \le {k_i}$ we have $$|\bv_{k} - \bv_{k_{i}}| \le |\bv_{k} - \bv_{k_{i-1}}| + |\bv_{k_i} - \bv_{k_{i-1}}| <   L t/\rho +  (L t/\rho + L),$$ hence $\mathscr{V}$ is contained in the $L(\frac{2 t}\rho + 1)$-neighborhood of $\mathscr{V}'$.
 \end{proof}}

\section{Proof of Theorem \ref{main}}\label{mainproof}

\subsection{\texorpdfstring{A passage from $F$ to $\mathscr{V}$}%
  {A passage from F to V}}
{Our proof of Theorem \ref{main}} is built on the  application of equidistribution of expanding translates of horospherical subgroups to Hausdorff dimension estimates of exceptional sets, an approach that goes back to \cite{KM}. However this is the first time, to the best of the authors' knowledge, when {these ideas are} applied to actions of sets more general than one-parameter semigroups.

\smallskip

In what follows "$\dim$" will stand for \hd. {We  claim that  in order to prove   Theorem \ref{main} it is enough to show that the following holds:
\begin{theorem}\label{reduction} For any $\varepsilon > 0$, any $x\in X$ and any non-empty open $W\subset U$  there exists a 
     bounded   $K\subset X$ and $t > 0$ such that  whenever $\mathscr{V}$ is a quasi-ray in ${\fs}_\fu^+$ with $\rho_{\fs_{\fu}^+}(\mathscr{V}) \ge t$, it holds that
     \eq{goal}{\dim\left(\big\{u\in W:  \exp(\mathscr{V})ux \subset K
\big\}\right) \ge \dim U - \varepsilon.}
\end{theorem}}
\begin{proof}[Proof of Theorem \ref{main} assuming Theorem \ref{reduction}]
{Recall that in  Theorem \ref{main} we are given a quasi-ray $F\subset G$ relative to $\fs$, where  $\fs$ is a  cone in $\fa$ such that $\overline{\fs} \subset {{\fs}_\fu^+} \cup \{ 0 \}$. 
Also recall that by definition  $F$ belongs to a bounded thickening of $\exp(\mathscr{V})$, where $\mathscr{V}$ is a quasi-ray in $\fs$;
the latter, in view of Lemma~\ref{from_s_to_sc}, amounts to saying that $\mathscr{V}$ is a quasi-ray in ${\fs}_\fu^+$ such that $\rho_{\fs_{\fu}^+}(\mathscr{V}) > 0$.
Further, using Proposition \ref{subray} one can assume that  $\rho_{\fs_{\fu}^+}(\mathscr{V})\ge t$ for any given $t > 0$ and still have $F$ contained in a bounded thickening of $\exp(\mathscr{V})$. This way one has $ux\in {\mathcal B}(F)$   whenever $\exp(\mathscr{V})ux \subset K$ for a bounded $K\subset X$; hence  the dimension estimate  \equ{goal} for any $\varepsilon,x$ and $W$ implies the thickness of the set \equ{conclU}.}

\smallskip
It remains to prove that the set ${\mathcal B}(F)$ {itself} is thick, {which is done by a standard fibering argument.} {Let us choose a  linear complement $\fv$ in $\fg$, that is, a subspace of $\fg$ with $\fg = \fu \oplus \fv$, and  a sufficiently small neighborhood $V_0 \subset \fv$ of $0$. Put
$O_V := \exp(V_0)$.
Then the map
$$\Psi: U \times O_V \rightarrow G,\ 
(u, g) \mapsto ug$$
has invertible differential at $(e, e)$. {Now take an arbitrary element $x \in X$. Then,  after shrinking $O_V$ and choosing a small neighborhood $O_U$ of $e\in U$, it follows that the projection map
$O_U \times O_V \rightarrow X$, $(u,g)\mapsto ugx$, is locally a bi-Lipschitz diffeomorphism.}
Let $$D = \left\{ (u,g) \in O_U \times O_V: ugx \in {\mathcal B}(F) \right\}.$$ Since Hausdorff dimension is preserved under bi-Lipschitz maps, it is enough to prove that {$D$ has full \hd. 
By the first part of the theorem,} 
for any $g \in O_V$,
the set 
$$
D \cap \left( O_U \times \{ g \} \right) = \{u\in 
O_U: ugx \in {\mathcal B}(F)\}
$$
has Hausdorff dimension equal to $\dim U$. Therefore, one can apply Marstrand's slicing lemma, e.g.\ {in the form of} \cite[Lemma 1.4(a)]{KM}, to conclude that
$$
{\dim D \geq \dim O_V + \dim U = \dim G = \dim  X}.
$$
This completes the proof of Theorem \ref{main} {modulo Theorem \ref{reduction}}}.
\end{proof}

\subsection{Equidistribution and return to {bounded} sets}

{As a preparation for} the proof {of Theorem \ref{reduction}, we} apply the equidistribution result from the previous section (Corollary \ref{cor:effective}) to characteristic functions of bounded subsets of $U$ and  $X$.

\begin{proposition}\label{thm22_cor}
    Let $B$ be a bounded subset of $U$ with $\mu_U(B) > 0$ and $\mu_U(\partial B) = 0$,   and let 
    $K$ be a bounded subset of $X$ with $\mu_X(K) > 0$ and $\mu_X(\partial K) = 0$. Also let ${Q}$ be a compact subset of $X$.
    Then for any $\delta > 0$ there exists ${t_1} = {t_1}(B,K, Q,\delta) > 0$, such that for any $x \in {Q}$ and any $\bv \in \fs_{\fu}^+$ with $\lfloor \bv\rfloor_{\fs_{\fu}^+} \geq {t_1}$  one has
    $$
    \left| \frac{\mu_U\left( \{ u \in B: \exp(\bv) ux \in K\} \right)}{\mu_U(B)} -  \mu_X(K)\right| < \delta.
    $$
\end{proposition}

\begin{proof} 
{We proceed in a standard way of approximation of characteristic functions of bounded subsets by smooth compactly supported functions.}
Let $\varepsilon > 0$. By $B_{\varepsilon}^+$ we will denote the $\varepsilon$-neighborhood of  {$B$ in $U$}, that is, $B_{\varepsilon}^+ = \{ u \in U: \dist(u,B) < \varepsilon \}$. Let us also put  $$B_{\varepsilon}^-: = U \smallsetminus \left( U \smallsetminus B \right)^+_{\varepsilon} = \{ u \in B: \dist(u, \partial B) \geq \varepsilon \}.$$ We define the {subsets} $K_{\varepsilon}^+$ and $K_{\varepsilon}^-$ {of $X$} analogously.

Fix $0 <\eta <1$, and define $\varepsilon = \varepsilon(\eta)$ in such a way that 
\begin{equation}\label{nu1}
\mu_U(B) - \mu_U(B_{2 \varepsilon}^-) < \eta, \ \mu_U(B_{2 \varepsilon}^+) - \mu_U(B)  < \eta
\end{equation}
and
\begin{equation}\label{nu2}
\mu_X(K) - \mu_X(K_{2 \varepsilon}^-) < \eta \ , \ \mu_X(K_{2 \varepsilon}^+) - \mu_X(K)  < \eta.
\end{equation}
{This can be done since $\mu_U(\partial B) = \mu_X(\partial K) = 0$.}
{Then  let $\gamma_{\varepsilon} \in C_c^{\infty}(U)$ and $\xi_{\varepsilon} \in C_c^{\infty}(G)$ be   nonnegative smooth functions whose supports are contained inside the $\varepsilon$-balls in $U$ (resp., $G$) around identity {and} such that $\int\limits_U \gamma_{\varepsilon} \,d \mu_U = \int\limits_G \xi_{\varepsilon} \,d \mu_G =1$.} 
It is known that such functions exist (see for example \cite[Lemma 2.4.7(b)]{KM}). 

Define functions $\varphi_{\varepsilon}^\pm: = \gamma_{\varepsilon} * 1_{B_{\varepsilon}^\pm} \in C_c^{\infty}(U)$ and $\psi_{\varepsilon}^\pm: = \xi_{\varepsilon} * 1_{K_{\varepsilon}^\pm} \in C_c^{\infty}(X)$. Here $f*g$ denotes the convolution of functions $f$ and $g$. It is easy to verify that:
\begin{itemize}
    \item $\supp(\varphi_{\varepsilon}^-) \subset {B}$, and if $u \in B_{2 \varepsilon}^-$, then $\varphi_{\varepsilon}^-(u) =1$;
    \item $\supp(\varphi_{\varepsilon}^+) \subset {B_{2 \varepsilon}^+}$, and if $u \in B$, then $\varphi_{\varepsilon}^+(u) =1$;
    \item $\supp(\psi_{\varepsilon}^-) \subset K$, and if $x \in K_{2 \varepsilon}^-$, then $\psi_{\varepsilon}^-(x) =1$;
    \item $\supp(\psi_{\varepsilon}^+) \subset {K_{2 \varepsilon}^+}$, and if $x \in K$, then $\psi_{\varepsilon}^+(x) =1$.
\end{itemize}

In particular, one has 
\begin{equation}\label{ineq_b}
1_{B_{2 \varepsilon}^-}(u) \leq \varphi_{\varepsilon}^-(u) \leq 1_B(u) \leq \varphi_{\varepsilon}^+(u) \leq 1_{B_{2 \varepsilon}^+}(u) \ \text{for any} \ u \in U
\end{equation} 
and 
\begin{equation}\label{ineq_k}
1_{K_{2 \varepsilon}^-}(x) \leq \psi_{\varepsilon}^-(x) \leq 1_K(x) \leq \psi_{\varepsilon}^+(x) \leq 1_{K_{2 \varepsilon}^+}(x) \ \text{for any} \ x \in X.
\end{equation}

Therefore, 

\begin{equation}\label{central}
\begin{aligned}\int\limits_U \varphi_{\varepsilon}^-(u) \psi_{\varepsilon}^-(\exp(\bv)ux) \,d \mu_U &\leq \mu_U\left( \{ u \in B: \exp(\bv) ux \in K\} \right) \\ &\leq \int\limits_U \varphi_{\varepsilon}^+(u) \psi_{\varepsilon}^+(\exp(\bv)ux) \,d \mu_U.\end{aligned}
\end{equation}

Now let us choose $M = \max \big( M({Q}, \varphi_{\varepsilon}^-, \psi_{\varepsilon}^-), M({Q}, \varphi_{\varepsilon}^+, \psi_{\varepsilon}^+)\big)$ where $M({Q}, \varphi, \psi)$ is defined in Corollary \ref{cor:effective}. Let $\zeta$ be as in Corollary \ref{cor:effective}, and pick $t_1 = t_1(B, K, {Q}, \eta)$ such that $M e^{-\zeta t_1} < \eta$. By Lemma \ref{b_doesnt_matter}, $\lfloor \exp(\bv)\rfloor_{S_U^+} = \lfloor \bv \rfloor_{\fs_\fu^+}$; thus, if $\lfloor \bv \rfloor_{\fs_\fu^+} \geq t_1$, by Corollary \ref{cor:effective} one gets
\begin{equation}\label{ineq-}
     \int\limits_U \varphi_{\varepsilon}^-(u) \psi_{\varepsilon}^-(\exp(\bv)ux) \,d \mu_U \geq  \int\limits_U \varphi_{\varepsilon}^- \,d \mu_U \int\limits_X \psi_{\varepsilon}^- \,d \mu_X    - \eta 
\end{equation}
and
\begin{equation}\label{ineq+}
     \int\limits_U \varphi_{\varepsilon}^+(u) \psi_{\varepsilon}^+(\exp(\bv)ux) \,d \mu_U \leq  \int\limits_U \varphi_{\varepsilon}^+ \,d \mu_U \int\limits_X \psi_{\varepsilon}^+ \,d \mu_X    + \eta.
\end{equation}
It follows from \eqref{ineq_b}, \eqref{ineq_k}, \eqref{nu1} and \eqref{nu2} that 
\begin{equation}\label{3eta-}
\begin{aligned}\int\limits_U \varphi_{\varepsilon}^- \,d \mu_U \int\limits_X \psi_{\varepsilon}^- \,d \mu_X    - \eta &\geq \left(\mu_U(B) - \eta \right) \left(\mu_X(K) - \eta \right) - \eta\\ & \geq \mu_U(B) \mu_X(K) - {\left( \mu_U(B) + 3 \right)} \eta
\end{aligned}\end{equation}
and
\begin{equation}\label{3eta+}
\begin{aligned}\int\limits_U \varphi_{\varepsilon}^+ \,d \mu_U \int\limits_X \psi_{\varepsilon}^+ \,d \mu_X    + \eta &\leq \left(\mu_U(B) + \eta \right) \left(\mu_X(K) + \eta \right) + \eta \\ &\leq \mu_U(B) \mu_X(K) + {\left( \mu_U(B) + 3 \right)} \eta.
\end{aligned}\end{equation}
Combining \eqref{central}, \eqref{ineq-}, \eqref{ineq+}, \eqref{3eta-} and \eqref{3eta+}, we conclude that 
$$
\begin{aligned} \mu_U(B) \mu_X(K) - {\left( \mu_U(B) + 3 \right)} \eta  & \leq \mu_U\left( \{ u \in B: \exp(\bv) ux \in K\} \right) \\ &\leq \mu_U(B) \mu_X(K) + {\left( \mu_U(B) + 3 \right)} \eta, \end{aligned}
$$
and thus
$$
    \left| \frac{\mu_U\left( \{ u \in B: \exp(\bv) ux \in K\} \right)}{\mu_U(B)} -  \mu_X(K)\right| < \frac{{\left( \mu_U(B) + 3 \right)} \eta}{\mu_U(B)}.
$$
Choosing $\eta \leq \frac{\delta \mu_U(B)}{{ \mu_U(B) + 3}}$ completes the proof.
\end{proof}

\subsection{\texorpdfstring{Metric properties of the conjugation by $\exp(\bv)$}%
  {Metric properties of the conjugation by exp(v)}}

For $\bv\in\fa$, let $\Omega_{\bv}: {H \rightarrow H}$ be the conjugation automorphism defined via $\Omega_{\bv}(g) = \exp(\bv) g \exp(-\bv)$. 
{Recall that the geometrically expanding cone $\fe^+_\fu$ is defined by the condition that $\Omega_{\bv}$ expand the metric on $U$.  
More precisely, let us denote by $\ell_{\bv}$ the Lipschitz constant of $\Omega^{-1}_{\bv}|_U$:
$$
    {\ell_{\bv}: = \sup\limits_{\substack{u_1, u_2 \in U \\ u_1 \neq u_2}}} \frac{\dist\big(\Omega_{\bv}^{-1}(u_1), \Omega_{\bv}^{-1}(u_2)\big)}{\dist(u_1,u_2)}, 
    $$
    and by $J_{\bv}$ the Jacobian of $\Omega_{\bv}|_U$:
    $$
    J_{\bv}: = \frac{\mu_U\big(\Omega_{\bv} ( B  )\big)}{\mu_U( B )}, 
    $$
    where the right hand side is independent of the choice of $B\subset U$ with $0 < \mu_U( B )< \infty$.}

\begin{lemma}\label{contraction bound lemma} {Let $c := 
    \min_{\lambda \in \Phi_{\fu}} |\lambda|$. Then: 
    \begin{enumerate}[label= \rm (\alph*)]
        \item\label{lem33_a} for any $\bv \in \fe_{\fu}^+$ one has $J_{\bv} \ge \exp\left(c \dim(U) \lfloor \bv\rfloor_{\fe_{\fu}^+}\right)$;
     \item\label{lem33_b} 
     there exists $\kappa >0$ (depending only on the choice of the  metric "$\dist$" on $U$)  such that for any $\bv \in \fe_{\fu}^+$ 
one has
    $$
    \ell_{\bv} \le \kappa \exp\left(-c \lfloor \bv\rfloor_{\fe_{\fu}^+}\right). 
    $$
    \end{enumerate}}
\end{lemma}

\begin{proof}
Fix $\bv \in \fe_{\fu}^+$ and let $\phi:=\Omega_{\bv}^{-1}|_U$. Since the differential 
 of $\phi$ at the identity can be written as  $${d\phi_e = \operatorname{Ad}\big(\exp(\bv)^{-1}\big)|_{\mathfrak u} = \exp\big(-\operatorname{ad}(\bv)|_{\mathfrak u}\big)},$$ on each root space  $\fh_\lambda$, where $\lambda \in \Phi_\fu$, 
it acts  
via
multiplication by $e^{-\lambda({\bv})}$.  
Moreover, 
{\eq{lambdaestimate}{{\lambda({\bv}) = |\lambda| \dist(\bv, \Ker \lambda) \geq c \dist(\bv, \partial \fe_{\fu}^+) = c \lfloor \bv\rfloor_{\fe_{\fu}^+}\text{ for any }\lambda \in \Phi_{\fu}.}} }
 In view of {the root space decomposition} \equ{rootdecomp},
it
implies that 
{$$\begin{aligned} J_{\bv} =  \det\left(d\phi_e\right)^{-1} = & \exp\left(\tr\operatorname{ad}(\bv)|_{\mathfrak u}\right) \\ =  \exp\left(\sum_{\lambda \in \Phi_{\fu}} \dim ({\fh_\lambda})\lambda({\bv})\right)\underset{\equ{lambdaestimate}}  \ge& \exp\left(c \lfloor \bv\rfloor_{\fe_{\fu}^+}\sum_{\lambda \in \Phi_{\fu}}\dim({\fh_\lambda})\right),\end{aligned}$$
proving \ref{lem33_a}. Similarly one can write}
$$\|d\phi_e\|_{\mathrm{op}} = \|\operatorname{Ad}(\exp(\bv)^{-1})|_{\fu}\|_{\mathrm{op}}
\le \kappa \exp\left(-c \lfloor \bv\rfloor_{\fe_{\fu}^+}\right),$$
where   $\| \cdot \|_{\mathrm{op}}$ is the operator norm,  and $\kappa$ depends only on the choice of the norm on $\fu$.

To prove \ref{lem33_b} it remains to show that
\begin{equation}\label{global_phi_estimate}
\operatorname{dist}\big(\phi(u_1),\phi(u_2)\big)
\le
\|d\phi_e\|_{\mathrm{op}}\operatorname{dist}(u_1,u_2)\text{ for any }u_1, u_2 \in U.
\end{equation}
Indeed, for $u\in U$  let $R_u:U\to U$ denote the right translation by $u$,
$R_u(h)=hu$. Since $\phi$ is an automorphism of $U$, one has
\[
\phi\circ R_u=R_{\phi(u)}\circ \phi.
\]
Differentiating at $e$, we get
\begin{equation}\label{diff_at_e}
d\phi_u\circ (dR_u)_e=(dR_{\phi(u)})_e\circ d\phi_e.
\end{equation}
Since $(dR_u)_e$ is a linear isomorphism between $\fu = T_eU$ and $T_uU$, every vector $v\in T_uU$ can be written uniquely as
$v=(dR_u)_e{w}$ with ${w}\in\mathfrak u=T_eU$. Using \eqref{diff_at_e}, we conclude that
\begin{equation}\label{phi_u_shifted}
d \phi_u(v) = d\phi_u\circ (dR_u)_e({w}) = (dR_{\phi(u)})_e\circ d\phi_e({w}).
\end{equation}
Since the Riemannian
metric on $U$ is right invariant, right translations are isometries on
tangent spaces; in particular, 
$$
|v|_u = |{w}| \ \ \ \ \text{and} \ \ \ \ |d \phi_u(v)|_{\phi(u)} = |d\phi_e({w})|,
$$
where the second equality follows from \eqref{phi_u_shifted}. Here we use the notation $| \cdot |_u$ for the norm on $T_u U$. Putting everything together, we get
$$
|d\phi_u(v)|_{\phi(u)} = |d\phi_e({w})| \leq \|d\phi_e\|_{\mathrm{op}}|{w}| = \|d\phi_e\|_{\mathrm{op}}|v|_u.
$$
Thus $\|d\phi_u\|_{\mathrm{op}}\leq \|d\phi_e\|_{\mathrm{op}}$
for every $u\in U$. Applying this inequality to the lengths of
piecewise $C^1$ curves and then taking the infimum over all curves
joining $u_1$ to $u_2$, we obtain \eqref{global_phi_estimate}. 
\end{proof}

\subsection{\texorpdfstring{Tessellations of $U$}{Tessellations of U}}Following \cite[\S3.1]{KM}, we say that an open set $B \subset U$ is a \underbar{tessellation domain} for the right action of $U$ on itself relative to a countable subset $\Lambda \subseteq U$ (we will just say "tessellation domain on $U$", and refer  to the pair $(B,\Lambda)$ as the tessellation of $U$) if

\begin{enumerate}
    \item $\mu_U\left( \partial B\right) = 0$; 
    \item If $\gamma_1 \neq \gamma_2$ are elements of $\Lambda$, then $B \gamma_1 \cap B \gamma_2 = \varnothing$;
    \item $U = \bigcup\limits_{\gamma \in \Lambda} \overline{B} \gamma$. 
\end{enumerate}

We will be using the following 

\begin{proposition}[\cite{KM}, Proposition 3.3] 
    For any $r > 0$ there exists a neighborhood $B_r$ of identity in $U$  such that $B_r$ is a tessellation domain on $U$ and $\diam(B_r) \leq
    r$. Moreover, {one} can choose the family $\{ B_r \}$ in such a way that if $r_1 < r_2$, then $B_{r_1} \subset B_{r_2}$.
\end{proposition}

The monotonic inclusion property is not stated in \cite{KM}, but it follows directly from the construction in their proof. From now on, $B_r$ is a fixed tessellation domain with the above properties, {and $\Lambda_r$ is the corresponding set of translations}.

{Using \cite[Proposition 3.4] {KM},  one can estimate the number of $\gamma \in \Lambda_r$ such that the right translate of $B_r$ by $\gamma$ is contained in $\Omega_{\bv} ( B_r)$ as follows:
$$
\# \left\{ \gamma \in \Lambda_r: B_r \gamma \subseteq \Omega_{\bv} ( B_r) \right\} \geq
{J_{\bv}}\Big( 1 - \frac{\mu_U\big(\{u\in U: \dist(u,\partial B_r)\le {\ell_{\bv}}\diam(B_r)\}\big)}{\mu_U(B_r)}
\Big).
$$}
{In view of} Lemma \ref{contraction bound lemma}(b)
{and} the assumption $\mu_U\left( \partial B_r\right) = 0$, we obtain the following

\begin{corollary}\label{cor_boundary_effect}
    Let $(B_r, \Lambda_r)$ be a tessellation of $U$ and $\delta > 0$. There exists ${t_2} = {t_2}( 
    r, \delta) > 0$ such that for any $\bv \in \fe_{\fu}^+
    $ {with $\lfloor \bv\rfloor_{{\fe_{\fu}^+}} \geq {t_2}$}
    one has 
    $$
    \# \{ \gamma \in \Lambda_r: \ B_r \gamma \subseteq  \Omega_{\bv}(B_r) \} \geq  
    {J_{\bv}}\big( 1 - \delta \big).
    $$
\end{corollary}

For $r>0$, $x\in X$, $K\subset X$ and $\bv \in {\fa}$ let us use the notation $$\Lambda_{r,\bv}(x, K): = \{ \gamma \in \Lambda_r: \ B_r \gamma \subseteq  \Omega_{\bv}(B_r) \ \text{and} \ B_r \gamma \exp(\bv) x \subset K \}.$$

\begin{proposition}\label{number_of_tiles_bound}
    Let $K \subset X$ be a compact set
    {such that} $\mu_X(\partial K) =0$, {and} let $\eta > 0$. {Then} there exists $r_0 = r_0(K, \eta)$ with the following property: for any $r \leq r_0$ there exists $t_0 = t_0({r, K, \eta})$ such that for any $x \in K$ 
    and any
    $\bv \in \fs_{\fu}^+
    $ {with $\lfloor \bv\rfloor_{\fs_{\fu}^+} \geq {t_0}$} one has 
\begin{equation}\label{number_of_tiles_bound_1}
    \#\Lambda_{r,\bv}(x, K)
    \geq 
    {J_{\bv}}\left( \mu_X(K) - \eta \right).
    \end{equation}
\end{proposition}

\begin{proof}
    The proof follows the lines of \cite[Proposition 3.6]{KM} and \cite[Proposition 4.1]{EKR25}. First of all, if $\mu_X(K) = 0$, \eqref{number_of_tiles_bound_1} holds automatically. From now on, assume $\mu_X(K) > 0$.

    {Recall that we are given  $\eta > 0$.} Choose a compact subset $K'$ contained in  $K$ such that 
    $\mu_X(\partial K') = 0$, $\dist(K', \partial K) > 0$ and $\mu_X(K') \geq \mu_X(K) - \frac{\eta}{3}$.
    Then choose  $r_0$ such that 
    $B_{r_0} B_{r_0}^{-1} K' \subseteq K$, which automatically implies that $B_{r} B_{r}^{-1} K' \subseteq K$ for any $r \leq r_0$.

    Fix $r \leq r_0$ and let $B = B_r, \ \Lambda = \Lambda_r$. Then, for any $x \in X$ {and}  any $\bv \in \fs_{\fu}^+ 
    $ 
    one has
    $$
    \begin{aligned}
        {\#\Lambda_{r,\bv}(x, K)} &\geq \# \{ \gamma \in \Lambda: \ B \gamma \subseteq  \Omega_{\bv}(B) \ \text{and} \ B \gamma \exp(\bv) x \subset B B^{-1} K' \} \\ &\geq \# \{ \gamma \in \Lambda: \ B \gamma \subseteq  \Omega_{\bv}(B) \ \text{and} \ \gamma \exp(\bv) x \in B^{-1} K' \} \\ &\geq \# \{ \gamma \in \Lambda: \ B \gamma \subseteq  \Omega_{\bv}(B) \ \text{and} \ B\gamma \exp(\bv) x \cap K' \neq \varnothing \} \\  = \# \{ \gamma \in \Lambda: \ B \gamma \subseteq  \Omega_{\bv}(B)\} &- \# \{ \gamma \in \Lambda: \ B \gamma \subseteq  \Omega_{\bv}(B) \ \text{and} \ B\gamma \exp(\bv) x \cap K' = \varnothing \}.
    \end{aligned}
$$
   {Since $(B,\Lambda)$ is a tessellation of $U$},  it follows
    that
    $$
    \begin{aligned}
    &\mu_U(B)   \# \{ \gamma \in \Lambda: \ B \gamma \subseteq  \Omega_{\bv}(B) \ \text{and} \ B\gamma \exp(\bv) x \,\cap\, K' = \varnothing \} \\
    & \leq \mu_U \big( \Omega_{\bv}(B) \smallsetminus \{ u \in \Omega_{\bv}(B): \ u \exp(\bv) x \in K' \} \big)\\
    & = 
    {J_{\bv}}\left( \mu_U(B) - \mu_U \big( \{ u \in B: \ \exp(\bv) u x \in K' \} \right)  \big).
    \end{aligned}
    $$

    Putting everything together, we conclude that
   $$
  \#\Lambda_{r,\bv}(x, K)   \geq \# \{ \gamma \in \Lambda: \ B \gamma \subseteq  \Omega_{\bv}(B)\} - 
  {J_{\bv}}\left( 1 - \frac{\mu_U \left( \{ u \in B: \ \exp(\bv) u x \in K' \} \right)}{\mu_U(B)}  \right).
    $$
    Now apply Proposition \ref{thm22_cor} and Corollary \ref{cor_boundary_effect}  with $\delta = \frac{\eta}{3}$: then for $\lfloor \bv\rfloor_{\fs_{\fu}^+} \geq \max(t_1, t_2)$ one has
   $$
    \#\Lambda_{r,\bv}(x, K) 
    \geq  
    {J_{\bv}}\left( 1 - \frac{\eta}{3} \right) - 
    {J_{\bv}}+ 
    {J_{\bv}}\left(\mu_X(K') - \frac{\eta}{3} \right)
     \geq
    {J_{\bv}}\left( \mu_X(K) - \eta \right),
   $$
{which finishes the proof.}     
\end{proof}

\subsection{Tree-like collections {and a \hd\ estimate}}\label{tree-like section}

{Our 
final goal in this section is, given a quasi-ray $\mathscr{V}$ in  {$\fs_{\fu}^+$  and $B = B_r$ as in the previous subsection, to construct a "Cantor-type" set $E_\infty$ that is contained inside   
$\{u\in B:  \exp(\mathscr{V})ux \subset K\}$ for some bounded $K\subset X$, and estimate its \hd. These  Cantor-type sets will be realized as limit sets of  strongly tree-like collection of subsets of $U$.
    Following \cite{KM}}, we call a collection $\mathcal{E}$ of subsets 
    of $U$ \underbar{strongly tree-like} if:
    \begin{enumerate}[label=(\arabic*)]
        \item\label{tl1} Each $E \in \mathcal{E}$ is a compact set with nonempty interior.
        \item\label{tl2} We have a partition
        \begin{equation*}
            \mathcal{E} = \bigcup_{i \in \Z_{\geq 0}} \mathcal{E}_k
        \end{equation*}
        with each $\mathcal{E}_i$ being finite and $\mathcal{E}_0$ being a singleton.
        \item\label{tl3} If $E' \in \mathcal{E}_i$ with $i \in \N$, there is a unique $E \in \mathcal{E}_{i-1}$ with $E' \subset E$.
        \item\label{tl4} For each $E \in \mathcal{E}_i$ with $i \in \Z_{\geq 0}$, there exists $E' \in \mathcal{E}_{i+1}$ with $E' \subset E$.
        \item\label{tl5} If $E_1, E_2 \in \mathcal{E}_i$ are distinct, then $\mu_U(E_1 \cap E_2)=0$. 
        \item\label{tl6} If we define 
        $d_i := \sup\{{\diam}(E) : E \in \mathcal{E}_i\}$, 
         then  $\lim_{i \to \infty} d_i = 0$.
    \end{enumerate}
    We write
    \begin{equation*}
        \cup \mathcal{E}_i := \bigcup_{E \in \mathcal{E}_i} E
    \end{equation*}
    and define the \underbar{limit set} of the collection to be $E_\infty := \bigcap_{i \in \N} \cup \mathcal{E}_i$.
}

It has been known since the work of McMullen \cite{Mc} and Urbanski \cite{U} that a lower estimate for the \hd\ of $E_\infty$ hinges on a quantitative strengthening of property (4) above (a uniform estimate for the number of children of any given set) and a rate of decay of $d_i$ in (6). Specifically, for $E \in \mathcal{E}_i$ let
\begin{equation*}
    \text{density}(\mathcal{E}_{i+1},E) := \frac{\mu_U\big((\cup \mathcal{E}_{i+1}) \cap E\big)}{\mu_U(E)} = \frac{\mu_U\big(\cup \mathcal{E}_{i+1}(E)\big)}{\mu_U(E)},
\end{equation*}
and for each $i \in \Z_{\geq 0}$,   set 
\begin{equation*}\label{eq: density}
        \Delta_i := \inf\left\lbrace \operatorname{density}(\mathcal{E}_{{i+1}}, E) : E \in \mathcal{E}_i \right\rbrace
\end{equation*}

We will use the following {well-known estimate}:
\begin{theorem}[Lemma 2.1 in \cite{U}]\label{thm: urbanski}
    For a strongly tree-like collection $\mathcal{E}$ and the resulting 
limit set $E_\infty$, {one has}
    \begin{equation*}
        \dim E_\infty \geq  \dim U - \limsup_{k \to \infty} \frac{\sum_{i=0}^{k} \log \Delta_i}{\log d_k}.
    \end{equation*}
\end{theorem}
\smallskip

{With this extensive preparation we are ready for the}

\begin{proof}[Proof of Theorem \ref{reduction}] {Recall that we are given $\varepsilon > 0$,  $x\in X$ and   non-empty open $W\subset U$. Picking $u\in W$ and replacing $x$ with $ux$, one can see that it is enough to prove the result for $W$  of the form $B_r$, where $r$ is small enough. Choose a compact set $K \subset X$ such that $0 < \mu_X(K) < 1$, $\mu_X(\partial K) = 0$ and $x \in K$. 
Then    put  $$\eta := \mu_X(K) - \mu_X(K)^2,$$ choose 
$r \le
r_0(K,\eta)
$, and take $t \ge t_0 = t_0(r, K,\eta)$. Our goal now is, given $\varepsilon > 0$, choose $t$ large enough so that for any quasi-ray $\mathscr{V}$   in $\fs_{\fu}^+$ with $\rho_{\fs_{\fu}^+}(\mathscr{V})\ge t$, the estimate \equ{goal} holds.}

{For a fixed} $r>0$, let $B = B_r$, and let $\Lambda = \Lambda_r$ be {a} discrete set  such that $B$ is a tessellation domain on $U$ relative to $\Lambda$. {Let us now use $\mathscr{V}$ as above and, following \cite{KM} where this was done for iterations of a single transformation, construct a tree-like collection of subsets of $\overline{B}$ as follows.}
We set $\mathcal{E}_0: = \{ \overline{B} \}$ and define collections $\mathcal{E}_i$ inductively: 
let 
$$
\mathcal{E}_1: = \left\{ \Omega_{{\bv_{1}}}^{-1}(\overline{B} \gamma): \gamma \in \Lambda_{r, {\bv_{1}}}(x, K) \right\}.
$$
Now suppose that
{for   $i \in \N$} the set $\mathcal{E}_i$ is defined, and for any  $E \in \mathcal{E}_i$ two properties hold:
\begin{enumerate}[label=(\Alph*)]
    \item\label{indA} $E = \Omega_{{\bv_{i}}}^{-1}  (\overline{B})  u$ for some (uniquely defined) $u \in U$, 
    and 
    \item\label{indB} ${\exp(\bv_{i})} E x \subset K$;
\end{enumerate}
{then we define}
$$
\mathcal{E}_{i+1}(E): = \left\{ \Omega_{{\bv_{{i+1}}}}^{-1}(\overline{B} \gamma)u: \gamma \in \Lambda_{r, {\bv_{{i+1}} - \bv_{i}}}({\exp(\bv_{i})}ux, K) \right\}
$$
and let
$$
\mathcal{E}_{i+1}: = \bigcup\limits_{E \in \mathcal{E}_i}\mathcal{E}_{i+1}(E).
$$
{Clearly properties \ref{indA} and \ref{indB} hold for $i=1$, and their validity} for $i$ replaced by $i+1$ follow immediately from construction.
{Also it can be easily seen that $\exp(\{\bv_1,\dots,\bv_i\})Ex\subset K$ for any $E \in \mathcal{E}_i$.}
Finally, let $E_\infty = \bigcap_{k \in \N} \cup \mathcal{E}_k$; {then it follows that  $\exp(\mathscr{V})E_\infty  x \subset K$.}

\smallskip

Let us verify that for small enough values of $r$ and large enough values of $t$ the collection $\mathcal{E}$ is strongly tree-like. Properties \ref{tl1}, \ref{tl2}, \ref{tl3} are automatic from {the} construction. Property \ref{tl5} follows directly since $B_r$ is a tessellation domain: we know that $\mu_U(\partial B_r) = 0$, and thus $\mu_U\left( \overline{B}_R \gamma_1 \cap \overline{B}_R \gamma_2 \right) = 0$ for any $\gamma_1 \neq \gamma_2 \in \Lambda_r$.

We will show that a stronger statement than \ref{tl4} holds if   {$t$ is} large enough. Take $E \in \mathcal{E}_i$, and let $u$ be as in \ref{indA}.
By Proposition \ref{number_of_tiles_bound} and with our choice of $r$ and $t$, 
for any $i\in\Z_{\ge 0}$ one has 
\begin{equation}\label{actual_number_of_tiles}
\# \Lambda_{r,{\bv_{i+1} - \bv_i}}\big({\exp(\bv_i)}  u x, K\big) \geq  J_{{\bv_{i+1} - \bv_i}}  \mu_X(K)^2,
\end{equation}
{as long as  $t \geq t_0$}.
{Note that, {by Lemma \ref{contraction bound lemma}(a),}
$$
\begin{aligned}
J_{\bv_{i+1} - \bv_i} \geq \exp\left(c \dim(U) \lfloor \bv_{i+1} - \bv_i \rfloor_{\fe_{\fu}^+}\right) &\geq \exp\left(c \dim(U) \lfloor \bv_{i+1} - \bv_i \rfloor_{\fs_{\fu}^+}\right)\\&\geq \exp\left(c \dim(U) \rho_{\fs_{\fu}^+}(\mathscr{V})\right) \geq e^{c    \dim(U)t},
\end{aligned}
$$ 
and, increasing ${t}$ if necessary, we can guarantee that $J_{{\bv_{i+1} - \bv_i}}  \mu_X(K)^2 \geq 1$.
}
This proves that the set {of children  $E'$ of an arbitrary set $E \in \mathcal{E}_i$ is nonempty, yielding \ref{tl4}. Moreover, we have the following density estimate: using \eqref{actual_number_of_tiles} and the fact that $\mu_U(\partial B) = 0$, we see that 
\begin{align*}
    \text{density}(\mathcal{E}_{i+1},E) &= \frac{\mu_U\big(\cup \mathcal{E}_{i+1}(E)\big)}{\mu_U(E)} = \# \Lambda_{r,\bv_{i+1} - \bv_i}(\exp(\bv_i) ux, K)   \frac{\mu_U(B_r)}{\mu_U\left( \Omega_{{\bv_{i+1} - \bv_i}}(B_r) \right)}\\  &= \frac{\# \Lambda_{r,\bv_{i+1} - \bv_i}(\exp(\bv_i) ux, K)}{J_{{\bv_{i+1} - \bv_i}}}   \geq \mu_X(K)^2,
\end{align*}
and thus \eq{bigdensity}{\Delta_i \geq \mu_X(K)^2\text{ for any }i \in \Z_{\ge 0}.}} 

It remains to show property \ref{tl6}. By Lemma \ref{contraction bound lemma}(b) we conclude that $d_{i+1} \leq \kappa e^{-ct} d_i$ for some constant $\kappa$ depending only on the choice of metric on $U$.
Therefore 
\begin{equation}\label{di_bound}
d_{i} \leq {\kappa^{i} e^{-ict}}{r}.
\end{equation}
{Increasing $t$   
if necessary}, we can guarantee that $\kappa e^{-ct} < 1$,
and thus $\lim\limits_{i \to \infty} d_i = 0$.

Applying Theorem \ref{thm: urbanski} and using \equ{bigdensity} along with \eqref{di_bound}, we obtain the lower bound 
$$
\dim E_\infty \geq  \dim U - \limsup_{k \to \infty} \frac{2(k+1) \log\mu_X(K)}{ \log r - k(ct - \log\kappa)} = \dim U - \frac{2 \log \mu_X(K)}{\log\kappa - ct},
$$
{hence for any $\varepsilon > 0$ one can choose $t\ge t_0$ such that \equ{goal} holds.} \end{proof}

\section{Quasimultiplicativity and {badly approximable matrices}}\label{qmd}

{The main goal of this section is to prove Theorem \ref{badisfull}.} 
Following \cite{N25}, 
{say that} an increasing function $h:
{I}\rightarrow \mathbb{R}_{>0}$, {where $I\subset  \mathbb{R}_{>0}$ is an interval, is} 
\underbar{quasimultiplicative} if 
there exist {$R_0 > 1$ and two constants $k_2 \geq k_1>0$} 
such that for any  $R \geq R_0$ and any {$T >0$ with $T,RT\in I$} one has
\begin{equation}\label{quasireformulation}
    R^{k_1} h(T) \leq h(RT) \leq R^{k_2} h(T).
\end{equation}

See \cite[\S5.1]{N25} for another equivalent definition. {Informally speaking}, the quasimultiplicativity {of $h$} is the property that allows {one} to turn multiplication of {the variable}  
by a constant into multiplication of {the value of $h$}
by a (different) constant, and vice versa, with some controlled error. 

\smallskip

We now return to the general set-up of approximation with weight functions, that is, to studying sets $\wab(f)$,  
where $f$ is an 
approximation function {and $\alpha_i:(0,1] \rightarrow (0,1]$ and $  \beta_j:[1,\infty) \rightarrow [1,\infty)$ are  monotonically increasing bijections satisfying \equ{normalizatiion}}. 
Some definitive results for Lebesgue-generic matrices have recently been obtained by the first-named author and Wang. The following is a direct corollary  of \cite[Theorem 2.7]{KW23}:

\begin{theorem}\label{badisnull}
    Suppose $\alpha_i$, $i=1,\dots,m$,  
    are quasimultiplicative functions, $\beta_j$ are arbitrary, and $f$ is 
    {an approximation function}. Then  the Lebesgue measure of $\wab(f)$ is zero or full according to {the convergence or divergence of} $\sum_{k=1}^\infty f(k)$.
\end{theorem}
In particular it follows that, under the quasimultiplicativity assumptions on $\alpha_i$, the set  \baab\ has Lebesgue measure zero.
This way we can view Theorem \ref{badisfull} as a statement regarding the Hausdorff dimension that is 
complementary to Theorem \ref{badisnull}.

Now let us describe the correspondence with dynamics. It will be convenient to use the variable $t = \log T$  and define {increasing bijections} $a_i,b_j:\R_{\ge0}\to \R_{\ge0}$ by
\begin{equation}\label{ab_definition}
e^{a_i(t)} = \frac{1}{\alpha_i \left(e^{-t}\right)} 
, \,\,\,\,\,\,\,\, e^{b_j(t)} = \beta_j(e^t).
\end{equation}
Clearly \equ{normalizatiion} implies that \eq{lognormalizatiion}{{\sum\limits_{i=1}^m a_i(t) = \sum\limits_{j=1}^n b_j(t) = t.}}
This way 
{the elements $g(t)$ of $F^{\alf, \betf}$ as in \equ{curveinT}
can be written in the form}
\eq{curveint}{
g(t): = \diag \left( e^{a_1(t)}, \ldots, e^{a_m(t)}, e^{-b_1(t)}, \ldots, e^{-b_n(t)}  \right) 
.
}
It is not hard to equivalently restate the quasimultiplicativity properties of  $\alpha_i$ and $\beta_j$ in terms of functions $a_i$ and $b_j$.

\begin{observation}\label{obs21}
    Suppose $\alpha_i, \beta_j$ and $a_i, b_j$ are related via \eqref{ab_definition}. Then the following are equivalent:
    \begin{itemize}
        \item[\rm (i)] all the functions $\alpha_i$ and $\beta_j$ are quasimultiplicative; \item [\rm (ii)] there exist ${\sigma} > 0$ and $t_0 > 0$ such that
\begin{equation}\label{nice}
    \frac{a_i(t + \tau) - a_i(t)}{\tau} \geq {\sigma}   \text{ and } \,   \frac{b_j(t + \tau) - b_j(t)}{\tau} \geq {\sigma}    \text{ for any }     {t>0} 
    , \ \tau \geq t_0
 \end{equation}
for any $i = 1, \ldots, m$ and $j = 1, \ldots, n$.
    \end{itemize}

\end{observation}

\begin{proof}
We will provide the proof for functions $\alpha_i$ and $a_i$; for $\beta_j$ and $b_j$ the proof is analogous.

    Suppose {$\alpha_i$ is quasimultiplicative; that is, there exist  {$R_0 > 1$ and   $k_1,k_2 >0$}  such that for any  $R \geq R_0$ and  $T,RT\in (0,1]$ one has $R^{k_1} \alpha_i(T) \leq \alpha_i(RT) \leq R^{k_2} \alpha_i(T)$.}
 Then {for any $t>0$ one has} 
    $$
    a_i(t+\tau) - a_i(t) = {-\log \alpha_i(e^{-(t+\tau)}) + \log \alpha_i(e^{-t}) = \log \left( \frac{\alpha_i(e^{ \tau}e^{-(t+\tau)})}{\alpha_i(e^{-(t+\tau)})} \right)} \geq \log e^{\tau k_1} = k_1 \tau,
    $$
    {hence} \eqref{nice} holds {for $a_i$} with ${\sigma} = k_1$ for $\tau \geq t_0: = \log R_0$. 

{To prove the converse, first} notice that  
{\eqref{nice} and \equ{lognormalizatiion} imply} that ${\sigma}n \leq 1$ and ${\sigma}m \leq 1$, thus {for any $i = 1, \ldots, m$ and $ j = 1, \ldots, n$
one has} \begin{equation}\label{ecin}
    \frac{a_i(t + \tau) - a_i(t)}{\tau} \leq 1-(m-1){\sigma} \,\,\,\, 
    \text{ for any }   {t >0},\ \tau \geq t_0.
\end{equation}
    Now suppose $a_i$ is such that {\eqref{nice} 
    holds. Take $R \geq R_0: = e^{t_0}$.} 
    Then for any $T,RT\in (0,1]$ one has
    $$
    \begin{aligned}\alpha_i(RT) = e^{-a_i(-\log R - \log T)} = e^{-a_i(-\log R - \log T) + a_i(-\log T) - a_i(-\log T)} \\  \geq e^{{\sigma} \log R} e^{-a_i(-\log T)} = R^{{\sigma}} \alpha_i(T).\end{aligned}
    $$
    By \eqref{ecin}  we get the complementary inequality $$\alpha_i(RT) \leq R^{1-(m-1){\sigma}} \alpha_i(T),$$ proving \eqref{quasireformulation} with $k_1 = {\sigma}$ and $k_2 = 1 - (m-1){\sigma}$.
\end{proof}

{Our next observation is that, under the quasimultiplicativity assumption on $\alpha_i$ and $\beta_j$}, 
 Theorem~\ref{main} can be applied {with $F = F^{\alf, \betf}$}. {Let us consider $U = \{ u_{\Theta}: \Theta \in \mr \}$, where $u_{\Theta}$ is as in \equ{defua}, and {put} $\fu = \Lie(U)$}.

\begin{observation}\label{obs22}
    Suppose 
    {the functions $\alpha_i$ and $\beta_j$ are quasimultiplicative.}
    Then there exists a cone $\fs$ with $\overline{\fs} \subset \fs_{\fu}^+ \cup \{ 0 \}$  such that ${F^{\alf, \betf}}$
    is a quasi-ray relative to $\fs$. 
\end{observation}

\begin{proof}
    { Recall from Example \ref{exampleua} that the cone  $\fs_{\fu}^+ = \fa_{\fu}^+ $ has the form \equ{aplus}. 
     {In other words,}}
    $$\fs_{\fu}^+ = \{ \pmb{y}:=\diag(y_1, \ldots, y_d) \in \fa: \lambda_i(\pmb{y}) > 0 \ \text{for} \ i = 1, \ldots, d  \}, $$
    where  $\fa$ is the diagonal subalgebra of $\mathfrak{sl}_d(\R)$ and {$$\lambda_i(\pmb{y}) = \begin{cases} y_i&\text{ if }  i =1, \ldots, m;\\ -y_{i}&\text{ if }  i =m+1, \ldots, d. \end{cases}$$} Note that for the norm $|\cdot|$ given by the Killing form on $\fa$ one has \eq{killingsld}{|\pmb{y}|^2 = 2d\sum_{i=1}^dy_i^2\  \Longrightarrow   \  |\lambda_{i}| = \frac{\sqrt{d-1}}{d\sqrt{2}}\text{ for any $i$}} 
    {(the implication above is an exercise left to the reader).}

     {Now let us switch from $\alpha_i,\beta_j$ to $a_i,b_j$ using \eqref{ab_definition}, and, using Observation \ref{obs21}, 
    take ${\sigma},t_0 > 0$ such that
\eqref{nice} holds
for any $i = 1, \ldots, m$ and $j = 1, \ldots, n$.    Then define $\mathscr{V} = (\bv_k)_{k\in\N}$, where
     $$\bv_k: = \diag\big(a_1(kt_0), \ldots, a_m(kt_0), -b_1(kt_0), \ldots, -b_n(kt_0)\big).$$
    Using \eqref{nice}, one gets
    $$
    \lambda_i(\bv_k - \bv_{k-1})  = a_i(kt_0) - a_i\big((k-1)t_0\big) \geq \sigma t_0$$
    and
    $$
    \lambda_{m+j}(\bv_k - \bv_{k-1})= b_j(kt_0) - b_j\big((k-1)t_0\big) \geq  \sigma t_0,
    $$
     {thus, in view of \equ{killingsld},} $\rho_{\fs_{\fu}^+}(\mathscr{V})\geq  \sigma t_0
    {\frac{d\sqrt{2}}{\sqrt{d-1}}}$. 
    Also 
    for any $k\in\N$ one has 
\begin{equation}\label{norm_small_increment}
\begin{aligned}
0 \le a_i(kt_0) - a_i\big((k-1)t_0\big) \le &\sum_{i=1}^m \Big(a_i(kt_0) - a_i\big((k-1)t_0\big)\Big) \\ \underset{\equ{lognormalizatiion}}= &t_0 \text{ for each } i = 1,\dots,m,    
\end{aligned}
    \end{equation}
    and similarly $b_j(kt_0) - b_j\big((k-1)t_0\big)\le t_0$ for each $j = 1,\dots,n$, which, again in view of \equ{killingsld}, implies that
    $|\bv_k - \bv_{k-1}| \leq d\sqrt2  t_0$. Hence, 
    by Lemma \ref{from_s_to_sc}, 
    there exists a cone $\fs$ with $\overline{\fs} \subset \fs_{\fu}^+ \cup \{ 0 \}$  such   that $\mathscr{V}$ is a quasi-ray in $\fs$.}

    {Now recall  that   $F^{\alf, \betf}$ can be written  in the form $\{g(t): t\ge 0\}$ as in 
\equ{curveint}. {Observe that, 
     as long as $(k-1)t_0 < t \leq kt_0$, by an argument similar to \eqref{norm_small_increment} one gets
     $$0 \leq a_i(kt_0) - a_i(t) \ \leq t_0 \ \text{for} \  i = 1, \ldots, m \ \text{and} \ 0 \leq b_j(kt_0) - b_j(t) \ \leq t_0 \ \text{for} \  j = 1, \ldots, n.$$
     Thus 
        $$\dist\big(\exp \bv_k, g(t)\big) =|\bv_k -\diag \big( {a_1(t)}, \ldots, {a_m(t)}, {-b_1(t)}, \ldots, {-b_n(t)}  \big) |\leq d \sqrt{2} t_0.$$ 
     This proves that $F^{\alf, \betf}$ is contained in the $(d \sqrt{2} t_0)$-neighborhood of $\exp(\mathscr{V})$.} 
     }
\end{proof}

Now we will show that, assuming the quasimultiplicativity of  all the weight functions,
$(\alf, \betf)$-bad approximability of {an $m\times n$ matrix} $\Theta$ is equivalent to a certain trajectory in {$X_d$} being bounded. The  {statement below}, which is a generalization of {the Dani} correspondence, directly follows (via the reparametrization described above) from \cite[Proposition 3.4]{N25}. We provide a separate proof for clarity and completeness.

\begin{proposition}\label{Dani_corr}
    Suppose all the functions $\alpha_i$ and $\beta_j$ are quasimultiplicative, {and let $g(t)$ be as in \equ{curveinT}}. Then $\Theta \in$ \baab\ if and only if the trajectory 
    $
    \left\{ g(t) {u_{\Theta} \Z^d}\right\}_{t \geq 0}
    $
    is bounded in $X_d$.
\end{proposition}

\begin{proof} 
{Since there are finitely many functions $\alpha_i$ and $\beta_j$, we will assume that {the constants $R_0,k_1,k_2$} 
are chosen in such a way that \eqref{quasireformulation} holds for $h$ replaced by any of these functions.}
    {By Mahler's Compactness Criterion (see \cite[Theorem IV, \S V.4.2]{Cas}),}
    the trajectory $\left\{ g(t) {u_{\Theta} \Z^d} \right\}_{t \geq 0}$ is bounded if and only if there exists $\delta > 0$ such that 
    $$
    \left|  {\begin{pmatrix}
       \frac1{\alpha_1(e^{-t})}
        &&&&& \\
        & \ldots &&&& \\
        && \frac1{\alpha_m(e^{-t})}
        &&& \\
        &&& \frac1{\beta_1(e^{t})}
        && \\
        &&&& \ldots & \\
        &&&&& \frac1{\beta_n(e^{t})}
    \end{pmatrix}}\begin{pmatrix}
    1_m & \Theta \\ & 1_n
\end{pmatrix} \begin{pmatrix}
    \p \\ \vq
\end{pmatrix} \right| > \delta
    $$
    for any $\p \in \Z^m, \vq \in \Z^n \nz$ {and large enough $t$.}  {Recall that the norm here is the supremum norm, and note that we do not need to care about the vectors for which $\vq = 0$, since for $\vq = 0$ and any nonzero $\p$ the norm above is at least $1$. Hence, with the substitution $T = e^t$ and 
    the notation $f_1(T) = 1/T$ defined in \eqref{Dirichlet_f1}}, the {above condition} is equivalent to the following: 
\begin{equation}\label{almost_bad}
{\begin{aligned}
\text{there exists $\delta > 0$  such that the system}\  \begin{cases}
        |\Theta_i \vq + \p| &\leq \delta \alpha_i\big(f_1(T)\big) \\ \qquad\ \, |\vq_j| &\leq \delta \beta_j(T)
    \end{cases} \\  \text{has no solutions $ \vq \in \Z^n \nz, \p \in \Z^m$  for all    large enough $T$.}\qquad
    \end{aligned}}
    \end{equation}

    Let $R \geq R_0$, and let $T' : = \frac{T}{R}$. Using \eqref{quasireformulation}, we obtain the following chains of inequalities:
\begin{equation}\label{beta_equiv}
    \delta R^{k_1} \beta_j(T') \leq \delta \beta_j(RT') = \delta \beta_j(T) \leq  \delta R^{k_2} \beta_j(T'),
    \end{equation}
    and, assuming that $\frac{1}{\delta^{{{1}/}{k_1}}}, \frac{1}{\delta^{{{1}/}{k_2}}} \geq R_0$,
\begin{equation}\label{alpha_equiv1}
        \delta \alpha_i \left( f_1(T) \right) = \delta \alpha_i \big( f_1(RT') \big) = \delta \alpha_i \left( \frac{1}{\delta^{{{1}/}{k_2}}} \cdot \frac{\delta^{{{1}/}{k_2}}}{R}f_1(T') \right) \leq  \alpha_i \left( \frac{\delta^{{{1}/}{k_2}}}{R}f_1(T') \right) 
    \end{equation}
    {and}
\begin{equation}\label{alpha_equiv2}
        \delta \alpha_i \big( f_1(T) \big) = \delta \alpha_i \big( f_1(RT') \big) = \delta \alpha_i \left( \frac{1}{\delta^{{{1}/}{k_1}}} \cdot \frac{\delta^{{{1}/}{k_1}}}{R}f_1(T') \right) \geq  \alpha_i \left( \frac{\delta^{{{1}/}{k_1}}}{R}f_1(T') \right).
    \end{equation}

    Suppose \eqref{almost_bad} holds; we can assume $\delta$ to be small enough, so that $\frac{1}{\delta^{{{1}/}{k_1}}} \geq R_0$. Choose $R \geq R_0$ large enough so that $\delta R^{k_1} \geq 1$. Using \eqref{beta_equiv} and \eqref{alpha_equiv2}, we conclude that the system of inequalities
    $$
    \begin{cases}
        |\Theta_i \vq + \p| &\leq   \alpha_i \left( \frac{\delta^{{{1}/}{k_1}}}{R}f_1(T') \right) \\ |\vq_j| &\leq \beta_j(T')
    \end{cases}
    $$
    has no solutions $ \vq \in \Z^n \nz$, $\p \in \Z^m$ for all  $T'$  large enough, that is, $\Theta$ is not $ \left(  \frac{\delta^{{{1}/}{k_1}}}{R} f_1, \alf, \betf \right)$-approximable, and in particular $\Theta \in$ \baab\ .
    
    Now suppose \eqref{almost_bad} does not hold. {Then} for any $\delta > 0$ such that $\delta R_0^{k_2} \leq 1$ and $\frac{1}{\delta^{{{1}/}{k_2}}} \geq R_0$ there exists an unbounded set of $T$ such that 
    $$
    \begin{cases}
        |\Theta_i \vq + \p| \leq  \delta \alpha_i\big(f_1(T)\big) \leq \alpha_i \left( \frac{\delta^{{{1}/}{k_2}}}{R_0}f_1(T') \right) \\ |\vq_j| \leq  \delta \beta_j(T) \leq  \delta R_0^{k_2} \beta_j(T') \leq \beta_j(T')
    \end{cases}
    $$
    where $T' = \frac{T}{R_0}$, and the last inequalities in each line follow from \eqref{beta_equiv} and \eqref{alpha_equiv1}. Therefore, for any $\delta > 0$ the matrix $\Theta$ is $\left(\frac{\delta^{{{1}/}{k_2}}}{R_0}f_1, \alf, \betf \right)$-approximable, and in particular \linebreak $\Theta \notin$ \baab.
\end{proof}

{It is worthwhile to point out that the} quasimultiplicativity {assumption} is necessary for Proposition \ref{Dani_corr} to hold, {as the next example indicates}.

\begin{example}
    Let 
    $$U = \left\{{u_{\Theta} = \begin{pmatrix}
        1_2 & \Theta \\   & 1_2
    \end{pmatrix}: \Theta   \in M_{2,2}
    }\right\}
    $$
    be a horospherical subgroup in $\SL_4(\R)$. Let $l(x): = 1 + \log x$, and take $$\alpha_1(x) = x l \left( \frac{1}{x} \right), \ \beta_1(T) = T/l(T)\quad \text{and}\quad \alpha_2(x) = \frac{1}{l \left( \frac{1}{x} \right)}, \ \beta_2(T) = l(T).$$ The functions $\alpha_2$ and $\beta_2$ are not quasimultiplicative. In this example  we have \linebreak $g(t) = \diag\left(\frac{e^t}{1+t}, 1+t, \frac{1+t}{e^t}, \frac{1}{1+t}\right)$. 

    Let $\theta_1, \theta_2$ be badly approximable numbers, and let $\Theta = 
    \diag(\theta_1, \theta_2)$. Then: 
    \begin{itemize}
        \item The trajectory $
    \left\{ g(t) u_{\Theta} \Z^4 \right\}_{t \geq 0}
    $
    is bounded in $X_4$. Indeed, the lattice $g(t) u_{\Theta} \Z^4$ is a direct product of two lattices: 
    $$
    \Lambda_1(t) = \begin{pmatrix}
        \frac{e^t}{1+t} &   \\   & \frac{1+t}{e^t}
    \end{pmatrix} \begin{pmatrix}
        1 & \theta_1 \\   & 1
    \end{pmatrix}\Z^2 \ \ \ \  \text{and} \ \ \ \ \Lambda_2(t) = \begin{pmatrix}
        1+t &   \\   & \frac{1}{1+t}
    \end{pmatrix} \begin{pmatrix}
        1 & \theta_2 \\   & 1
    \end{pmatrix}\Z^2.
    $$
    Since $\theta_1$ and $\theta_2$ are both badly approximable, {the} trajectories $\Lambda_1(t)$ and $\Lambda_2(t)$ are bounded, hence their product is bounded as well.
    
    \item $\Theta \notin$ \baab. Let us fix ${\varepsilon} > 0$ and show that $\Theta$ is $({\varepsilon}f_1, \alf, \betf)$-approximable. Note that the system of inequalities \eqref{Dirichlet_hom_fcns} with $f = {\varepsilon}f_1$ decouples into two independent systems
    \begin{equation}\label{sl4_example_eq}
    \begin{cases}
        |\theta_1 q_1 + p_1| \leq \frac{{\varepsilon}}{T} l \left( \frac{T}{{\varepsilon}} \right), \\
        |q_1| \leq \frac{T}{l(T)}
    \end{cases} \ \ \ \ \text{and} \ \ \ \ \begin{cases}
        |\theta_2 q_2 + p_2| \leq \frac{1}{l\left( \frac{T}{{\varepsilon}} \right)} = \frac{1}{l(T) - \log {\varepsilon}}, \\
        |q_2| \leq l(T).
    \end{cases}
    \end{equation}

    Let $(-p_k, q_k)$ be the sequence of best approximations for $\theta_2$; then, 
    $$
    |\theta_2 q_k + p_k| < \frac{1}{q_{k+1}} \ \ \ \ \text{and} \ \ \ \ q_{k+1} - q_k \rightarrow \infty.
    $$
    Let $T_k : = l^{-1}(q_k)$. Suppose $k$ is large enough to guarantee that $q_{k+1} - q_k > - \log {\varepsilon}$. Then we have $$|\theta_2 q_k + p_k| \leq \frac{1}{q_{k+1}} \leq \frac{1}{q_k - \log {\varepsilon}} = \frac{1}{l(T_k) - \log {\varepsilon}} =  \frac{1}{l\left(  {T_k}/{{\varepsilon}} \right)}.$$ Since $|q_k| \leq l(T_k) = q_k$, we conclude that for any large enough $k$ and $T = T_k$ the pair $(p_2, q_2) = (-p_k, q_k)$ is a solution for the second system in \eqref{sl4_example_eq}. It is clear that the trivial pair   $(p_1, q_1) = (0,0)$ is a solution for the first system for any $T$. Therefore, for any large enough $k$ and $T = T_k$ there exists a nonzero solution $(p_1, p_2, q_1, q_2) = (0, -p_k, 0, q_k)$ for the combined system \eqref{sl4_example_eq} of inequalities, proving that $\Theta$ is $({\varepsilon}f_1, \alf, \betf)$-approximable.
    \end{itemize}  
\end{example}

\smallskip

{We conclude the section by observing  that} Theorem \ref{badisfull} directly follows from Theorem \ref{main} via 
Observation \ref{obs22} and Proposition \ref{Dani_corr}.

\begin{proof}[Proof of Theorem \ref{badisfull}]
    If all the functions $\alpha_i$ and $\beta_j$ are quasimultiplicative, 
    Observation~\ref{obs22} yields that the set $F^{\alf, \betf}$ is a quasi-ray relative to some cone $\fs$ such that $\overline{\fs} \subset \fs_{\fu}^+ \cup \{ 0 \}$. By Theorem \ref{main}, the set 
\begin{equation}\label{thick_for_matrices}
    \left\{ u \in U: u\Z^d \in {\mathcal B}(F^{\alf, \betf}) \right\}
    \end{equation}
    is thick in $U$. {Consequently,} {since the map $\mr \rightarrow U, \ \Theta \mapsto u_{\Theta}$ is locally bi-Lipschitz, the set $\{ \Theta \in \mr: u_{\Theta} \in \eqref{thick_for_matrices}  \}$ is thick in $\mr$}. It remains to use Proposition \ref{Dani_corr} to notice that $u_{\Theta}$ belongs to the set \eqref{thick_for_matrices} if and only if $\Theta \in$ \baab.
\end{proof}

\appendix
\section{More about expanding cones}\label{cone_section}

In the Appendix we take $\fu$ to be a Lie algebra of a horospherical subgroup $U = U_+(b)$ of a connected semisimple group $H$, let $\fa = \Lie(A)$, where $A$ is a maximal $\Ad$-diagonalizable subgroup of $H$ containing  {$b$, and} discuss  properties of algebraically/geometrically expanding and super-expanding cones $\fa_{\fu}^+$, $\fe_{\fu}^+$ and  $\fs_{\fu}^+$ in more detail. 

\subsection{Simplified presentation of cones}

Recall that the cones $\fe_{\fu}^+$ and $\fa_{\fu}^+$ were defined using  the root system
$\Phi_\fu : = \{ \lambda \in \Phi: \fh_{\lambda} \subseteq \fu \}$, 
    see \eqref{au_big} and \eqref{eu}.
    It will be useful to simplify the definition of both cones by throwing away unnecessary roots. Namely, 
    let ${\Delta_\fu} \subseteq \Phi_\fu $ be the (uniquely defined) minimal set of additive generators of $\Phi_\fu $, and recall that for $\lambda\in\Phi_\fu  $ we defined a corresponding element $s_\lambda\in \fa$ via \eqref{s_alpha_def}. Let us now show that in  \eqref{au_big} and \eqref{eu} one can replace $\Phi_\fu  $ by $\Delta_\fu$. 
    \begin{lemma}\label{minimal_general} One has
\begin{equation}\label{a_e_minimal}
        \fe_{\fu}^+: = \{ {\bv} \in \fa: \lambda({\bv}) > 0 \,\, \text{for any} \,\, \lambda \in {\Delta_\fu} \} \ \ \ \ \text{and} \ \ \ \ \fa_{\fu}^+: = \left\{ \sum\limits_{\lambda \in {\Delta_\fu}} t_{\lambda}{\bv}_{\lambda}: t_{\lambda} > 0 \right\}.
    \end{equation}
    \end{lemma}

\begin{proof}
    The inclusion 
    $$\{ {\bv} \in \fa: \lambda({\bv}) > 0 \,\, \text{for any} \,\, \lambda \in \Phi_{\fu} \} \subseteq \{ {\bv} \in \fa: \lambda({\bv}) > 0 \,\, \text{for any} \,\, \lambda \in {\Delta_\fu} \}$$
    follows from the inclusion ${\Delta_\fu} \subseteq \Phi_\fu $, and the opposite inclusion holds since every element in $\Phi_{\fu}$ is a linear combination of elements in ${\Delta_\fu}$ with nonnegative coefficients (some of which are strictly positive). This  proves the 
    $\fe_{\fu}^+$-part of \eqref{a_e_minimal}.
    
    For $\fa_{\fu}^+$, the proof is also straightforward. If $\lambda = \mu + \nu$, then ${\bv}_{\lambda} = {\bv}_{\mu} + {\bv}_{\nu}$. Therefore the inclusion
    $$
    \left\{ \sum\limits_{\lambda \in \Phi_\fu } t_{\lambda}{\bv}_{\lambda}: t_{\lambda} > 0 \right\} \subseteq \left\{ \sum\limits_{\lambda \in {\Delta_\fu}} t_{\lambda}{\bv}_{\lambda}: t_{\lambda} > 0 \right\}
    $$
    holds. To prove the opposite inclusion, let $s = \sum_{\lambda \in {\Delta_\fu}} t_{\lambda}{\bv}_{\lambda}$, where {the} coefficients $t_{\lambda}$ are positive. Fix $\lambda_0 \in \Phi_\fu $, and let $\lambda_0 = \sum\limits_{i=1}^r c_i \lambda_i$ be the representation of $\lambda_0$ as a linear combination of elements of ${\Delta_\fu}$ with strictly positive integer coefficients (we only enumerate the elements of $\Delta_{\fu}$ that occur in this decomposition). Fix a positive $ \varepsilon < \min _{i=1 \ldots, r} \frac{t_{\lambda_i}}{c_i}$. Then 
    $$
    \bv = \sum\limits_{
    \lambda \in {\Delta_\fu},\, \lambda \notin \{ \lambda_1, \ldots, \lambda_r \}
    } t_{\lambda} {\bv}_{\lambda} + \sum\limits_{i=1}^r (t_{\lambda_i} - c_i \varepsilon) \bv_{\lambda_i} + \varepsilon \bv_{\lambda_0}
    $$
    is a presentation of $\bv$ as a linear combination of elements ${\bv}_{\lambda}$ for $\lambda \in {\Delta_\fu} \cup \{ \lambda_0 \}$ with positive coefficients. Repeating this procedure for every $\lambda_0 \in \Phi_\fu  \smallsetminus {\Delta_\fu}$, we obtain a presentation of the form $\bv = \sum\limits_{\lambda \in \Phi_\fu } t'_{\lambda}{\bv}_{\lambda}$ where $t'_{\lambda} > 0$, thus $\bv \in \left\{ \sum\limits_{\lambda \in \Phi_\fu } t_{\lambda}{\bv}_{\lambda}: t_{\lambda} > 0 \right\}$.
\end{proof}

\subsection{\texorpdfstring{Cones in $\SL_d(\R)$}{Cones in SL(d,R)}}

In this subsection we describe all possible algebraically and geometrically expanding cones in the case
 $H = \SL_d(\R)$ and $A = $  the subgroup of diagonal matrices in $H$.  We let  $\mathfrak{a} = \Lie(A)$; that is, 
\begin{equation*}\label{fa_diag}
\fa = \{ {\pmb{y}:=\diag(y_1, \ldots, y_d)}: y_1 + \ldots + y_d = 0 \}.
\end{equation*}
Any horospherical subgroup of $H$ can be conjugated to a subgroup of the group of upper-triangular unipotent matrices. A convenient way to describe all horospherical upper-triangular subgroups is to associate them with partitions of $\{1,\dots,d\}$.
Namely, for any $U$ as above there exists  a partition
\begin{equation}\label{block_structure}
{\mathcal{I} = \{ I_1, \ldots, I_s \}, \text{ where } I_{k} = \{ i_{k-1} + 1, \ldots, i_{k} \}\text{ for } k = 1,\dots,s,\ i_0 = 0,}
\end{equation}
such that {$I_1 \sqcup \ldots \sqcup I_s = \{1,\ldots,d\}$, and}
\eq{Ubypartition}{
U = \left\{
\begin{pmatrix}
    1 & u_{1,2} & 
    \ldots & u_{1,d} \\
    0 & 1 &  
    \ldots & u_{2,d} \\
    \ldots & \ldots 
    & \ldots & \ldots \\
    0 & 0 & 
    \ldots & 1
\end{pmatrix}:  u_{i,j} = 0   \text{ if } i, j \in I_k  \text{ for the same } k.
\right\} 
}
Conversely, for any such partition one can construct $U$ as described above. From now on, we parameterize horospherical subgroups by partitions.   {Here one can recall the two examples from \S\ref{def_cones_section}:
\begin{itemize}
    \item the minimal case (Example \ref{exampleua}), where $U$ was given by a minimal non-trivial partition $\{1,\dots,m\} \sqcup \{m+1,\dots,d\}$, and \item the maximal case 
(Example \ref{examplemax}), where $U$ corresponded to the maximal partition $\{1\} \sqcup\ldots \sqcup\{d\}$.
\end{itemize}} 

{In general, when $U$ is of the form \equ{Ubypartition} for a partition $\mathcal{I}$, we will denote by $\fa_{\mathcal{I}}^+$, $\fe_{\mathcal{I}}^+$ and $\fs_{\mathcal{I}}^+$ the corresponding expanding cones. Our plan in this section is to describe them explicitly.} 
\smallskip 

{It is easy to see that for $U$ as in \equ{Ubypartition}} one has
$$
\Phi_\fu  = \{ \lambda_{i,j}: i \in I_k, j \in I_l \,\, \text{for some} \, k < l \}
$$
and
$$ {\Delta_\fu} = \{ \lambda_{i,j}: i \in I_k, j \in I_{k+1} \,\, \text{for some} \, k = 1, \ldots, s-1\},
$$
where $\lambda_{i,j}({\pmb{y}}) = y_i- y_j$.
{Also} let $\bv_{i,j} \in \mathfrak{a}$ be the $d \times d$ diagonal matrix with $1$ in the $i$-th position, $-1$ in the $j$-th position and $0$ elsewhere. Then, {in view of \equ{killingsld},} for any ${\pmb{y}} \in \fa$ one has $\lambda_{i,j}({\pmb{y}}) = {\frac{1}{2d}}\langle \bv_{i,j}, {\pmb{y}}\rangle$; {that is, $\bv_{i,j}$ is proportional to $\bv_{\lambda_{i,j}}$ defined via \eqref{s_alpha_def}}. 
{Now,  using \eqref{au_big}, \eqref{eu} and Lemma \ref{minimal_general},} one 
can explicitly describe $\fa_{\mathcal{I}}^+$ and $\fe_{\mathcal{I}}^+$ as follows:
\begin{equation}\label{a_by_combinations_eq}
    \fa_{\mathcal{I}}^+: = \left\{ \sum\limits_{k=1}^{s-1}\sum\limits_{(i, j) \in I_{k} \times I_{k+1}} t_{i,j} \bv_{i,j} : \,\, t_{i,j} > 0 \right\}
    \end{equation}
    and
   \begin{equation}\label{e_sld_min}
    \fe_{\mathcal{I}}^+: = \left\{ 
    {\pmb{y}}: y_i>y_j \, \text{if } i \in I_k, j \in I_{k+1}  \text{ for some } k= 1, \ldots, s-1; \,\, \sum\limits_{i=1}^d y_i = 0 \right\}.
    \end{equation}

{It will be instructive to convert 
\eqref{a_by_combinations_eq} to the dual form, that is, represent the cone 
$\fa_{\mathcal{I}}^+$ as the positivity set of certain linear functionals on $\fa$.}

\begin{lemma}\label{a_plus_in_coord}
     {Let $\mathcal{I}$ be as in \eqref{block_structure}. Then one has} 
\begin{equation}\label{form_with_primes}
    \fa_{\mathcal{I}}^+  = \left\{ 
    {\pmb{y}}\left| \begin{aligned} \sum\limits_{i \in I_1 \sqcup \ldots \sqcup I_{k-1}} y_i + \sum\limits_{i \in I_k} \min(y_i,0) > 0 \,\, \text{for any} \ 2 \leq k \leq s-1; \\ y_i > 0 \ \text{for } i \in I_1; \ y_i < 0 \ \text{for } i \in I_s; \ \sum\limits_{i=1}^{d} y_i = 0 \qquad \end{aligned} \right. \right\}\end{equation}
\end{lemma}

{To illustrate this lemma, we {again} recall Examples \ref{exampleua} and \ref{examplemax}. In Example \ref{exampleua} 
{$s=2$, hence}  the set $\{ k: 2 \leq k \leq s-1 = 1 \}$ is empty. {Therefore the conditions defining $\fa_{\mathcal{I}}^+$ reduce to the second line in \eqref{form_with_primes}, i.e.\ to} \eqref{ex22_cone_a}.}
{{Now look at   Example \ref{examplemax}:} there one has
{$s = d$ and $I_k = \{k\}$ for $k = 1,\dots,s$. Thus the first line in  \eqref{form_with_primes} is equivalent to 
$$
\sum\limits_{i =1}^{k-1} y_i +  \min(y_k,0) > 0 \iff \sum\limits_{i =1}^{k-1} y_i  > 0 \text{ and } \sum\limits_{i =1}^{k} y_i  > 0
$$
for any $2 \leq k \leq d-1$. This shows that \eqref{form_with_primes} is equivalent to \eqref{ex23_cone_a}.}
}

\begin{proof}[Proof of Lemma \ref{a_plus_in_coord}.]
    
    First, suppose that $
    {\pmb{y}}$ belongs to the set \eqref{form_with_primes}. Let us define:
    \begin{itemize}
        \item $\varepsilon_k: = \dfrac1{|I_k|}\left(\sum\limits_{i \in I_1 \sqcup \ldots \sqcup I_{k-1}} y_i + \sum\limits_{j \in I_k} \min(y_j,0)\right)
        $ for $k = 1, \ldots, s$. We know that $\varepsilon_k > 0$ for $k = 2, \ldots, s-1$, and $\varepsilon_1 = \varepsilon_s = 0$.
        \item $z_j = -\min(y_j,0) + \varepsilon_k$ for every $j \in I_k$ and any $k = 1, \ldots, s$. Note that all $z_j$ except for $j \in I_1$ are strictly positive, and $z_j = 0$ for $j \in I_1$.
        \item $x_i = y_i + z_i$ for $i \in \{ 1, \ldots, d\}$. One can also check that all $x_i$ are strictly positive except for $i \in I_s$, and $x_i=0$ for $i \in I_s$. 
    \end{itemize}

    A useful observation is that for any $k  = 1, \ldots, s-1$ one has
\begin{equation}\label{two_equal_sums}
        \sum\limits_{j \in I_{k+1}} z_j = \sum\limits_{i \in I_1 \sqcup \ldots \sqcup I_{k}} y_i = \sum\limits_{u \in I_k} x_u.
    \end{equation}

    For $(i, j) \in I_{k}\times I_{k+1}, \ k = 1, \ldots, s-1$,  we define $t_{i,j} : = \frac{x_i z_j}{\sum_{{\ell} \in I_{k}}x_{\ell}}$. Note that $t_{i,j} > 0$. We claim that 
       $ \sum\limits_{k=1}^{s-1}\sum\limits_{(i, j) \in I_{k} \times I_{k+1}} t_{i,j} {\bv}_{i,j} = 
       {\pmb{y}}$.
    Indeed: if $i \in I_k$, {then} the $i$-th coordinate of the vector $\sum\limits_{k=1}^{s-1}\sum\limits_{(i, j) \in I_{k} \times I_{k+1}} t_{i,j} {\bv}_{i,j}$ equals to {$  \sum\limits_{j \in I_{k+1}} t_{i, j}-\sum\limits_{{\ell} \in I_{k-1}} t_{{\ell}, i} $}, where we are using the convention that $I_0 = I_{s+1} = \varnothing$. Note that
    $$
    -\sum\limits_{{\ell} \in I_{k-1}} t_{{\ell}, i} = -\frac{\sum_{{\ell} \in I_{k-1}}x_{\ell} z_i}{\sum_{{\ell} \in I_{k-1}}x_{\ell}} = -z_i \,\,\,\, \text{if $k > 1$ and} \,\, -\sum\limits_{{\ell} \in I_{k-1}} t_{{\ell}, i} = 0 = -z_i \,\,\,\, \text{if $k = 1$,}
    $$
    and
    $$
    \sum\limits_{j \in I_{k+1}} t_{i, j} = \frac{\sum_{j \in I_{k+1}}x_i z_j}{\sum_{{\ell} \in I_{k}}x_{\ell}} = x_i \,\,\,\, \text{if $k < s$ and} \,\, \sum\limits_{j \in I_{k+1}} t_{i, j} = 0 = x_i \,\,\,\, \text{if $k = s$,}
    $$
    where the simplification in the last line follows from \eqref{two_equal_sums}. Summing up, we have shown that the $i$-th coordinate is equal to $x_i - z_i = y_i$; 
    thus $
    {\pmb{y}}\in \fa_{\mathcal{I}}^+$.

    \smallskip

    Now suppose that $
    {\pmb{y}}= \sum\limits_{k=1}^{s-1}\sum\limits_{(i, j) \in I_{k} \times I_{k+1}} t_{i,j} {\bv}_{i,j}$ for some $t_{i, j} > 0$; we will prove that all the inequalities in \eqref{form_with_primes} hold. Indeed: note that $y_i = - \sum\limits_{u \in I_{l-1}} t_{u,i} + \sum\limits_{j \in I_{l+1}} t_{i, j}$ for any $i \in I_l$. Thus
   $$\begin{aligned}
    \sum\limits_{i \in I_1 \sqcup \ldots \sqcup I_{k-1}} y_i = \sum\limits_{l=1}^{k-1} \left(  - \sum\limits_{u \in I_{l-1}} \sum\limits_{i \in I_{l}} t_{u,i} + \sum\limits_{i \in I_{l}} \sum\limits_{j \in I_{l+1}} t_{i, j} \right)\\ = - \sum\limits_{l=1}^{k-2}\sum\limits_{(i, j) \in I_l \times I_{l+1}} t_{i,j} + \sum\limits_{l=1}^{k-1}\sum\limits_{(i, j) \in I_l \times I_{l+1}} t_{i,j} = \sum\limits_{(i, j) \in I_{k-1} \times I_{k}} t_{i,j}.
    \end{aligned}$$

    Let $J  \subseteq I_k$ be some subset of indices. Then 
    $$
\begin{aligned}\sum\limits_{i \in I_1 \sqcup \ldots \sqcup I_{k-1} \sqcup J} y_i &= \sum\limits_{(i, j) \in I_{k-1} \times I_{k}} t_{i,j} - \sum\limits_{(i, j) \in I_{k-1} \times J} t_{i,j} + \sum\limits_{(i, j) \in J \times I_{k+1}} t_{i,j}  \\ &= \sum\limits_{(i, j) \in I_{k-1} \times \left( I_{k} \smallsetminus J \right)} t_{i,j} + \sum\limits_{(i, j) \in J \times I_{k+1}} t_{i,j},\end{aligned}
    $$
    and the latter expression is:
    \begin{enumerate}[label=\rm (\alph*)]
        \item\label{aplusb} strictly positive if $k = 1$ for any nonempty $J$ (since $I_{k+1} = I_2$ is nonempty);
        \item\label{aplusa} strictly positive if $k = 2, \ldots, s-1$ since the sets $I_{k-1}, I_{k+1}$ and one of the sets $J$ and $I_k \smallsetminus J$ are nonempty;
        \item\label{aplusc} strictly positive if $k = s$ and $J \neq I_s$ and {equal to} zero if $J = I_s$ since $I_{k+1} = I_{s+1} = \varnothing.$
    \end{enumerate}

{It remains to make correct choices of $J$ to verify all the inequalities in \eqref{form_with_primes}:
\begin{itemize}
\item For $i \in I_1$, consider $J = \{ i \}$. By \ref{aplusb}, we have established that $y_i > 0$. 
\item Next, consider $2 \leq k \leq s-1$, and let $J$ be a subset of $I_k$ consisting of indices of all the non-positive elements (potentially empty). Then $\sum\limits_{i \in J} y_i =  \sum\limits_{i \in I_k} \min(y_i,0)$, and \ref{aplusa} implies that $$\sum\limits_{i \in I_1 \sqcup \ldots \sqcup I_{k-1}} y_i + \sum\limits_{i \in I_k} \min(y_i,0) > 0.$$
\item Finally, for $i \in I_s$, consider $J = I_s \smallsetminus \{ i \}$. Then $\sum_{j \in I_1 \sqcup \ldots \sqcup I_{k-1} \sqcup J} y_j =   -y_i$, and \ref{aplusc} yields that the latter is positive, proving that $y_i < 0$.
\end{itemize}
We have shown that all the inequalities from \eqref{form_with_primes} hold, which completes the proof.   
}
\end{proof}

\subsection{\texorpdfstring{Comparing $\fa_{\fu}^+$ with $\fe_{\fu}^+$}%
  {Comparing the algebraic and geometric cones}}\label{abelian}

{Next we would like to study the position of algebraically and geometrically expanding cones with respect to one another:}  when the inclusion $\fa_{\fu}^+ \subseteq \fe_{\fu}^+$ holds, when the opposite inclusion $\fe_{\fu}^+ \subseteq \fa_{\fu}^+$ holds, and when neither of the cones is inside the other one. {
Let us start with a simple criterion for the inclusion  $\fa_{\fu}^+ \subseteq \fe_{\fu}^+$ to hold that can be written down for cones in an arbitrary connected semisimple group $H$.}

\begin{proposition}\label{inner_product_yields_inclusion}
    The inclusion $\fa_{\fu}^+ \subseteq \fe_{\fu}^+$ holds if and only if ${\langle \lambda, \mu \rangle} \geq 0$ for any $\lambda, \mu \in \Phi_\fu $. 
\end{proposition}

\begin{proof}
    Suppose ${\langle \lambda, \mu \rangle} \geq 0$ for any $\lambda, \mu \in \Phi_\fu $. Fix $\bv = \sum_{\mu \in \Phi_\fu } t_{\mu} {\bv}_{\mu} \in \fa_{\fu}^+$, where $t_{\mu} > 0$. Our goal is to show that $\lambda({\bv}) > 0$ for any $\lambda \in \Phi_\fu $, which would imply that $\bv \in \fe_{\fu}^+$. Indeed:
    $$
    \lambda({\bv}) = \sum\limits_{\mu \in \Phi_\fu } t_{\mu} \lambda({\bv}_{\mu}) \stackunder[1pt]{{}={}}{\scriptstyle  \eqref{s_alpha_def}} t_{\lambda} \langle \lambda, \lambda \rangle + \sum\limits_{\lambda \neq \mu \in \Phi_\fu } t_{\mu} \langle \lambda, \mu \rangle \geq t_{\lambda} \langle \lambda, \lambda \rangle > 0.
    $$

    Now suppose $\lambda, \mu \in \Phi_\fu $ are such that ${\langle \lambda, \mu \rangle} < 0$. Fix a constant $\xi > \frac{\left|\sum_{\nu \in \Phi_\fu , \, \nu \neq \mu} \langle \lambda, \nu \rangle \right|}{\left| \langle \lambda, \mu \rangle \right|}$, and let $\bv: = \xi {\bv}_{\mu} + \sum_{\nu \in \Phi_\fu , \, \nu \neq \mu} {\bv}_{\nu}$; by definition, ${\bv} \in \fa_{\fu}^+$. On the other hand, 
    $$
    \lambda({\bv}) = \xi  \langle \lambda, \mu \rangle + \sum\limits_{\nu \in \Phi_\fu , \ \nu \neq \mu} \langle \lambda, \nu \rangle < - \left| \sum\limits_{\nu \in \Phi_\fu , \ \nu \neq \mu} \langle \lambda, \nu \rangle \right| + \sum\limits_{\nu \in \Phi_\fu , \ \nu \neq \mu} \langle \lambda, \nu \rangle < 0,
    $$
    therefore $\bv \notin \fe_{\fu}^+$ and $\fa_{\fu}^+ \nsubseteq \fe_{\fu}^+$.
\end{proof}

{We have seen in Example \ref{exampleua} that {for $H=\SL_d(\R)$} the inclusion $\fa_{\fu}^+ \subseteq \fe_{\fu}^+$ holds for  {abelian (equivalently, minimal)} $\fu$. It turns out that  {in general} this inclusion
is also related to the minimality and commutativity of $\fu$, 
 however these relations are one-sided. Namely, we have the following two statements:
}

\begin{proposition}\label{aine_general}
    If $\fu$ is abelian, then $\fa_{\fu}^+ \subseteq \fe_{\fu}^+$. That is, $\fs_{\fu}^+ = \fa_{\fu}^+$.
\end{proposition}

\begin{proposition}\label{incl_mininal}
    If 
    the inclusion $\fa_{\fu}^+ \subseteq \fe_{\fu}^+$ holds, then $\fu$ is minimal within the class of horospherical Lie algebras such that their projection  to each of the simple factors
    of $\fh$
is nontrivial.
\end{proposition}

In order to prove Proposition \ref{aine_general}, we will need 
{the following}

\begin{lemma}\label{crit_of_abelian}
    $\fu$ is abelian if and only if 
    \begin{equation}\label{colsed_under_addition}
        \text{for any} \,\, \lambda, \mu \in \Phi_\fu  \,\, \text{one has} \,\, \lambda + \mu \notin \Phi.
    \end{equation}
\end{lemma}

\begin{proof}
    Since $\fu = \bigoplus_{\lambda \in \Phi_\fu } \fh_{\lambda}$, $\fu$ is abelian if and only if $[ \fh_{\lambda}, \fh_{\mu}] = 0$ for any $\lambda, \mu \in \Phi_\fu $ (including $\lambda = \mu$).
    Let us notice that $\lambda + \mu \neq 0$ for any $\lambda, \mu \in \Phi_\fu $: indeed, there exists $\bv \in \fa$ on which every element of $\Phi_\fu $ is positive, and in particular $\lambda(\bv), \mu(\bv) > 0$. Since $\lambda + \mu \neq 0$, it is known (see for instance \cite[Theorem]{KraljevicRootSubspaces})
    that $[\fh_{\lambda}, \fh_{\mu}] = \fh_{\lambda + \mu}$ if $\lambda + \mu \in \Phi$ and $[\fh_{\lambda}, \fh_{\mu}] = 0$ otherwise. This completes the proof.
\end{proof}

\begin{proof}[Proof of Proposition \ref{aine_general}.] Suppose $\fu$ is abelian; by Lemma \ref{crit_of_abelian}, it implies that \eqref{colsed_under_addition} holds. By Proposition \ref{inner_product_yields_inclusion}, it is enough to show that $\langle \lambda, \mu \rangle \geq 0$ for any $\lambda, \mu \in \Phi_\fu $.
    Since $U$ is horospherical with respect to an element of $A$,
there exists $\bv\in \fa$ such that 
$$
\lambda(\bv)>0 \ \ \text{for every }\lambda\in\Phi_{\fu}.
$$
In particular, if $\lambda,\mu\in\Phi_{\fu}$, then
$(\lambda+\mu)(\bv)>0$,
so $\lambda+\mu\ne0$.

Suppose for the sake of contradiction that $\langle \lambda,\mu\rangle<0$. The set of restricted roots $\Phi$ forms an abstract root system (\cite[Corollary 6.53]{Knapp}). By \cite[Proposition 2.48(e)]{Knapp}, if $\langle \lambda,\mu\rangle<0$, then
$\lambda+\mu$ is either a root or $0$. The second possibility has
just been excluded, hence $\lambda+\mu\in\Phi$, contradicting \eqref{colsed_under_addition}.
Therefore $\langle\lambda,\mu\rangle\ge0$, which completes the proof.
\end{proof}

\begin{proof}[Proof of Proposition \ref{incl_mininal}.]
{Let $\fh = \fh_1 \oplus \ldots \oplus \fh_m $ be a  decomposition of  $\fh$ as a direct sum of simple Lie algebras. Let $\fu = \fu_1 \oplus \ldots \oplus \fu_m$ be the corresponding decomposition of $\fu$; we note that $\fu$ is minimal in the desired class if and only if each of $\fu_i$ is minimal in $\fh_i$. It can be seen from the definitions that $\fa_{\fu}^+ = \fa_{\fu_1}^+ \oplus \ldots \oplus \fa_{\fu_m}^+$ and $\fe_{\fu}^+ = \fe_{\fu_1}^+ \oplus \ldots \oplus \fe_{\fu_m}^+$; therefore, it is enough to prove the statement for minimal horospherical $\fu$ in a simple Lie algebra. From now on, $\fh$ is assumed to be simple.
}
    
    Assume that $\fa_{\fu}^{+}\subseteq \fe_{\fu}^{+}$. Suppose $U = \exp(\fu)$ is an expanding horospherical subgroup with respect to $b=\exp(\bv)$, and choose a positive root system $\Phi_+$ of $\fh$
    such that $\lambda(\bv) \ge 0$ for any $\lambda \in \Phi_+$. Let $\Delta$ be the set of simple roots in $\Phi_+$. It is known that $\fu$ is minimal  if and only if $|\Phi_{\fu} \cap \Delta| = 1$. 
    
    Suppose, for the sake of contradiction, that $\{ \lambda, \mu \} \subseteq \Phi_{\fu} \cap \Delta$. Since $\fh$ is simple, the Dynkin diagram of $\Phi_+$ is connected. Let $\xi_0 = \lambda$, $\xi_1, \ldots, \xi_l = \mu$ be a sequence of simple roots along the shortest path in the Dynkin diagram connecting $\lambda$ and $\mu$; in particular, \begin{equation}\label{pairwise_products}
        \langle \xi_i, \xi_j \rangle = 0 \ \text{if} \ |i - j| \geq 2 \ \text{and} \ \langle \xi_i, \xi_j \rangle < 0 \ \text{if} \ |i - j| =1.
    \end{equation}

    Now define the sequence of 
    {functionals} $\nu_k, \ k = 0, \ldots, l,$ via $\nu_k: = \xi_0 + \ldots + \xi_k$. We will show by induction that $\nu_k \in \Phi_{\fu}$ for any $k = 0, \ldots, l$.

    Since $\nu_0 = \lambda$, there is nothing to prove for $k= 0$. Now suppose $\nu_{k-1} \in \Phi_{\fu}$. By \eqref{pairwise_products}, $$\langle \nu_{k-1} , \xi_k \rangle = \sum\limits_{i=0}^{k-1} \langle \xi_{i} , \xi_k \rangle = \langle \xi_{k-1} , \xi_k \rangle < 0.$$
    {Then one}  can apply \cite[Proposition 2.48(e)]{Knapp} and, since $\nu_{k-1} + \xi_k \neq 0$, conclude that $\nu_k = \nu_{k-1} + \xi_k \in \Phi$. Finally, $\nu_k(\bv) = \xi_0(\bv) + \ldots + \xi_k(\bv) \geq \xi_0(\bv) = \lambda(\bv) > 0$ and $\nu_k \in \Phi_{\fu}$.
    It remains to notice that $\nu_{l-1}$ and $\mu$ are two elements of $\Phi_{\fu}$ for which $\langle \nu_{l-1}, \mu\rangle = \langle \nu_{l-1}, \xi_l \rangle < 0$, thus by Proposition \ref{inner_product_yields_inclusion} one has $\fa_{\fu}^+ \nsubseteq \fe_{\fu}^+$. This finishes the proof via contradiction.
\end{proof}

{The aforementioned equivalence between minimality and commutativity for}
horospherical subgroups 
{of $\SL_d(\R)$  implies the following}
simple criterion: 

\begin{corollary}\label{a_in_e_sl}
  {Let $\fu$ be a horospherical upper-triangular subalgebra of $\mathfrak{sl}_d(\R)$ corresponding to a partition $\mathcal{I} = \{ I_1, \ldots, I_s \}$ of $\{1,\dots,d\}$.  Then}  the inclusion $\fa_{\mathcal{I}}^+ \subseteq \fe_{\mathcal{I}}^+$ holds if and only if $s = 2$; that is, in the situations featured in Example \ref{exampleua}.
\end{corollary}

{One may ask if either of the Propositions \ref{aine_general} and \ref{incl_mininal} can be turned into a criterion for an arbitrary 
$H$. In both cases, the answer is negative, as shown in the examples below.}

\begin{example}[$\fa_{\fu}^+ \subseteq \fe_{\fu}^+$ for a non-abelian $\fu$]\rm
    Let $H$ be a real Lie group of type $B_2/C_2$, e.g. ${ \rm Sp}_4(\R)$, let $\lambda $ be its short simple root and $\mu$ be its long  simple root (see \cite[Appendix C, \S 1]{Knapp} for this root system realization).
    The set 
\[
\Phi_{\fu}
=
\{\lambda,\lambda+\mu,2\lambda+\mu\}
\]
is an ideal of the set of positive roots with respect to $\lambda$ and $\mu$, thus 
$\fu = \bigoplus_{\nu \in \Phi_\fu } \fh_{\nu}$ is a {horospherical subalgebra} in $\fh$. It is not abelian by Lemma \ref{crit_of_abelian}, since $\lambda + (\lambda + \mu) = 2\lambda+\mu$ is a root. 
    
  Now let us show that $\fa_{\fu}^+ \subseteq \fe_{\fu}^+$.  The Cartan matrix of $B_2$ is 
$$
\begin{pmatrix}
    2 & -2 \\ -1 & 2
\end{pmatrix}
$$
(see \cite[\S 11.1]{H72}), thus $\langle \lambda, \mu \rangle = - \langle \lambda, \lambda \rangle = - \frac{1}{2} \langle \mu, \mu \rangle$. Therefore the pairwise scalar products of elements of $\Phi_{\fu}$ are
$$
\langle \lambda, \lambda + \mu \rangle = \langle \lambda, \lambda \rangle + \langle \lambda, \mu \rangle = 0 \geq 0; \ \ \ \langle \lambda, 2\lambda + \mu \rangle = 2\langle \lambda, \lambda \rangle + \langle \lambda, \mu \rangle = \langle \lambda, \lambda \rangle > 0;
$$
and 
$$
\langle \lambda + \mu, 2\lambda + \mu \rangle = 2\langle \lambda, \lambda \rangle + 2\langle \lambda, \mu \rangle + \langle \lambda, \mu \rangle + \langle \mu, \mu \rangle  = \frac{1}{2}\langle \mu, \mu \rangle > 0.
$$
By Proposition \ref{inner_product_yields_inclusion}, $\fa_{\fu}^+ \subseteq \fe_{\fu}^+$. 
\end{example}

\begin{example}[Inclusion $\fa_{\fu}^+ \subseteq \fe_{\fu}^+$ fails for a minimal $\fu$] \rm 
Let $H$ be the split real simple group of type $G_2$. {Let $\lambda,\mu$ be the simple roots of $G_2$, with $\lambda$ the short root and $\mu$ the long root. It is known (see again \cite[\S 11.1]{H72}) that the Cartan matrix of $G_2$ is 
$$
\begin{pmatrix}
    2 & -3 \\ -1 & 2
\end{pmatrix},
$$ thus $2 \langle \lambda, \mu \rangle = -3 \langle \lambda, \lambda \rangle$ and
$\langle\lambda,\lambda+\mu\rangle
=
\langle\lambda,\lambda\rangle+\langle\lambda,\mu\rangle
=-\frac{1}{2}\langle\lambda,\lambda\rangle <0.$
The set 
\[
\Phi_{\fu}
=
\{\lambda,\lambda+\mu,2\lambda+\mu,
3\lambda+\mu,3\lambda+2\mu\}
\]
is a minimal ideal of the set of positive roots with respect to $\lambda$ and $\mu$, {since it only contains one simple root $\lambda$}. Thus 
$\fu = \bigoplus_{\nu \in \Phi_\fu } \fh_{\nu}$ is a minimal  {horospherical subalgebra} in $\fh$. Since $\langle\lambda,\lambda+\mu\rangle < 0$, by Proposition \ref{inner_product_yields_inclusion} one has $\fa_{\fu}^{+}\nsubseteq \fe_{\fu}^{+}$.}
\end{example}

\medskip

{We conclude the Appendix by formulating the criterion for the opposite inclusion \linebreak $\fe_{\fu}^+ \subseteq \fa_{\fu}^+$ to hold. For simplicity, we will only do it in the case $H = \SL_d(\R)$.} 
We will use the following notation: if $\mathcal{I} = \{ I_1, \ldots, I_s \}$ is a partition of $\{1,\dots,d\}$, by $|I_k|$ we denote the number of elements in the block $I_k$; {that is, $|I_k| = i_k - i_{k-1}$ in the notation from \eqref{block_structure}}.

\begin{proposition}\label{eina_prop}
    The inclusion $\fe_{\mathcal{I}}^+ \subseteq \fa_{\mathcal{I}}^+$ holds if and only if the following two conditions are satisfied:
    \begin{enumerate}[label= \rm (\roman*)]
        \item\label{eina_A} $|I_1| = |I_s| = 1$; and 
        \item\label{eina_B} $|I_k|^2 \leq 1 + 4 
        i_{k-1}(d-i_k)
        $ 
        for any $k \in \{ 2, \ldots, s-1 \}$.
    \end{enumerate}
\end{proposition}

\medskip

{In particular, when $d=4$ it follows that the partition $\{1\}, \{2,3\},\{4\}$ gives rise to the inclusion $\fe_{\mathcal{I}}^+ \subseteq \fa_{\mathcal{I}}^+$, while partitions  $\{1,2\}, \{3\},\{4\}$ and $\{1\}, \{2\},\{3,4\}$ produce examples of algebraically and geometrically expanding cones that are not contained in one another. Of course this can also be verified directly by comparing \eqref{form_with_primes} with \eqref{e_sld_min}.}

\smallskip
We start proving Proposition \ref{eina_prop} with the following technical lemma.

\begin{lemma}\label{sums_of_blocks}
    Suppose $
    {\pmb{y}}\in \fe_{\mathcal{I}}^+$. {Then} for any $k = 1, \ldots, s-1$ one has
    $$
    \sum\limits_{i \in I_1 \sqcup \ldots \sqcup I_k} y_i > 0.
    $$
\end{lemma}
\begin{proof}
    Let $y : = {\frac1{i_k}{\sum\limits_{i \in I_1 \sqcup \ldots \sqcup I_k} y_i}}$ (where $i_k = |I_1| + \ldots + |I_k|$, as before) be the average of the elements over first $k$ blocks and $y^* : = {\frac1{d - i_k}{\sum_{i \in I_{k+1} \sqcup \ldots \sqcup I_s} y_i}}$ be the average over the remaining blocks.
    Since, {in view of \eqref{e_sld_min},} the inequality $y_i > y_j$ holds for any pair $(i, j) \in \left( I_1 \sqcup \ldots \sqcup I_k \right) \times \left( I_{k+1} \sqcup \ldots \sqcup I_s \right)$, we see that $y > y^*$. Next, 
    $$
    0 = \sum\limits_{i = 1}^d y_i = \sum\limits_{i \in I_1 \sqcup \ldots \sqcup I_k} y_i + \sum\limits_{i \in I_{k+1} \sqcup \ldots \sqcup I_s} y_j = i_k y + (d-i_k)y^* \,\,\,\, \Rightarrow \,\,\,\, y > 0 > y^*,
    $$
    and thus $\sum\limits_{i \in I_1 \sqcup \ldots \sqcup I_k} y_i = i_k y > 0$.
\end{proof}

\begin{proof}[Proof of Proposition \ref{eina_prop}.]
First, suppose that conditions \ref{eina_A} and \ref{eina_B} hold, and let
$
{\pmb{y}}\in \mathfrak e_I^+$. 
Since $|I_1|=|I_s|=1$, it follows that $y_i>0$ for
$i\in I_1$ and $y_i<0$ for $i\in I_s$.

Now fix $2 \leq k \leq s-1$, and put $\ell:=
i_{k-1}, \
r:=|I_k|$ and $v:=
d-i_k$.
Condition \ref{eina_B} now takes the form $r^2 \leq 4\ell v + 1$.

\begin{observation}\label{obs1_eina}
    One has $r^2 \leq 4\ell v + 1$ if and only if $\left\lfloor\frac{r^2}{4}\right\rfloor\leq \ell v$.
\end{observation}
\begin{proof}
    The "only if" direction is trivial. For the "if" direction, note that the inequality $\left\lfloor\frac{r^2}{4}\right\rfloor\leq \ell v$ implies that $r^2 \leq 4\ell v + 3$, and since any square of an integer is congruent to 0 or 1 modulo 4, the latter yields the inequality $r^2 \leq 4\ell v + 1$.
\end{proof}

\begin{observation}\label{obs2_eina}
    For $r \geq 2$, one has 
    \begin{equation*}
        \max\limits_{1 \leq a \leq r-1} a(r-a) = \left\lfloor\frac{r^2}{4}\right\rfloor.
    \end{equation*}
\end{observation}
\begin{proof}
   It is easy to check that the inequality $a(r-a) \leq \frac{r^2}{4}$ holds for any $1 \leq a \leq r-1$. To obtain the desired equality, one needs to take $a = \lfloor \frac{r}{2} \rfloor$, which implies $r-a = \lceil \frac{r}{2} \rceil$.
\end{proof}

By Lemma \ref{a_plus_in_coord}, in order to show that $
{\pmb{y}}\in \fa_{\mathcal{I}}^+$ one needs to prove the inequality
\begin{equation}\label{middle_k}
\sum_{i\in I_1\sqcup\cdots\sqcup I_{k-1}}y_i
+
\sum_{i\in I_k}\min(y_i,0)>0.
\end{equation}
If all coordinates in $I_k$ are nonnegative, \eqref{middle_k} follows from Lemma \ref{sums_of_blocks} applied to $k-1$. If all coordinates in $I_k$ are nonpositive, it
follows from Lemma \ref{sums_of_blocks} applied to $k$.

It remains to consider the case when $a$ coordinates in $I_k$ are negative,
where $1\leq a\leq r-1$ (this can only happen when $r \geq 2$). 
Let
        $$
        {y_L} : = \frac{y_1 + \ldots + y_\ell}{\ell}, \,\,\,\, {y_+} : = \frac{1}{r-a} \sum\limits_{i \in I_k} \max( y_i, 0) ,
        $$
        $$
        {y_-}: = \frac{1}{a}\sum\limits_{i \in I_k} \min( y_i, 0) \,\,\,\, \text{and} \,\,\,\, {y_R}: = \frac{y_{d-v + 1} + \ldots + y_d}{v}.
        $$
It follows from Lemma \ref{sums_of_blocks} and \eqref{e_sld_min} that 
\begin{equation*}\label{avg_ineqs}
            {y_L} > {y_+} \geq 0 > {y_-} > {y_R}.
\end{equation*}
Since $
{\pmb{y}}\in \fa$, we have
\begin{equation*}\label{two_stars}
\ell {y_L}+a{y_-}+(r-a){y_+}+v{y_R}=0.
\end{equation*}
Using the new notation, we rewrite \eqref{middle_k} as
\begin{equation}\label{new_two_stars}
    \ell {y_L} + a{y_-} > 0 \ \iff \ (r-a){y_+} + v{y_R} < 0.
\end{equation}

Set $\kappa:=-\frac{{y_R}}{{y_L}}>0$. 
In view of Observation \ref{obs1_eina} and Observation \ref{obs2_eina}, the inequalities
$
\ell<\kappa a
$ and $
a<r-\kappa v
$
cannot hold simultaneously, since they would imply
\[
\frac{\ell}{a}<\kappa<\frac{r-a}{v}
\quad\Longrightarrow\quad
\ell v<a(r-a).
\]
If $\ell\geq\kappa a$, then, since ${y_-}>{y_R}=-\kappa {y_L}$,
\[
\ell {y_L}+a{y_-}>\ell {y_L}+a{y_R}={y_L}(\ell-\kappa a)\geq0,
\]
proving \eqref{new_two_stars}.
If $a\geq r-\kappa v$, then $r-a\leq\kappa v$, and, since ${y_+}<{y_L}$,
\[
(r-a){y_+}+v{y_R}
<
(r-a){y_L}-\kappa v{y_L}
=
{y_L}(r-a-\kappa v)
\leq0,
\]
again proving \eqref{new_two_stars}. Therefore $
{\pmb{y}}\in \fa_{\mathcal{I}}^+$, proving that conditions \ref{eina_A} and \ref{eina_B} are sufficient for the inclusion $\fe_{\mathcal{I}}^+ \subseteq \fa_{\mathcal{I}}^+$ to hold.

\medskip

Now let us prove that conditions \ref{eina_A} and \ref{eina_B} are necessary. 
    Suppose $|I_1| \geq 2$. Let $y_1 = 1$ and choose $y_2, \ldots, y_d$ to be  strictly negative 
    {with} $y_2 > \ldots > y_d$ and $y_2 + \ldots + y_d = -1$. Then $
    {\pmb{y}}\in \fe_{\mathcal{I}}^+\smallsetminus  \fa_{\mathcal{I}}^+$ since $y_2 \in I_1$ and $y_2 < 0$.
 {An} analogous argument works in the case $|I_s| \geq 2$.
{This shows the necessity of condition \ref{eina_A}.}

Now suppose that condition \ref{eina_B} fails for some
$2 \leq k \leq s-1$. As before, let $\ell:=
i_{k-1}, \
r:=|I_k|$ and $v:=
d-i_k$;
then
\[
\left\lfloor\frac{r^2}{4}\right\rfloor>\ell v.
\]
By Observation \ref{obs2_eina}, there exists $a\in\{1,\ldots,r-1\}$ such that
\[
a(r-a)=\left\lfloor\frac{r^2}{4}\right\rfloor.
\]
Thus, we can choose $\kappa > 0$ such that 
\[
\frac{\ell}{a}<\kappa<\frac{r-a}{v}.
\]
Take $\varepsilon>0$ such that
\[
\varepsilon<
\min\left\{
\kappa-\frac{\ell}{a},
\,1-\frac{\kappa v}{r-a}
\right\},
\]
{and} choose strictly decreasing sequences
\[
\varepsilon_1>\cdots>\varepsilon_\ell,
\qquad
\delta_1>\cdots>\delta_v,
\]
such that
\[
\sum_{i=1}^\ell \varepsilon_i=0,\qquad
\sum_{j=1}^v\delta_j=0,\qquad
|\varepsilon_i|,|\delta_j|<\varepsilon.
\]
{Now} let 
$$
y_i:=1+\varepsilon_i \ \text{for} \ i = 1, \ldots, l; \ y_{l+r+j}:=-\kappa+\delta_j \ \text{for} \ j = 1, \ldots, v;
$$
$$
y_{\ell+1} = \ldots = y_{\ell+a} :=-\frac{\ell}{a}, \ \text{and} \ y_{\ell+a+1} = \ldots = y_{\ell+r}: =\frac{\kappa v}{r-a}.
$$
It is easy to check using \eqref{e_sld_min} that $ 
{\pmb{y}}\in \mathfrak e_{\mathcal{I}}^+$.
On the other hand,
\[
\sum_{i\in I_1\sqcup\cdots\sqcup I_{k-1}}y_i
+
\sum_{i\in I_k}\min(y_i,0)
=
\ell+a\left(-\frac{\ell}{a}\right)
=
0,
\]
so \eqref{middle_k} fails, and thus $
{\pmb{y}}\notin \fa_{\mathcal{I}}^+$.
This shows the necessity of \ref{eina_B} and completes the proof.
\end{proof}


\end{document}